\documentclass[12pt]{amsart}
\usepackage[margin=1in]{geometry}
\usepackage[british]{babel}

\usepackage[utf8]{inputenc}
\usepackage{amsmath,amsfonts,amssymb,amsthm,stmaryrd}
\usepackage[all,cmtip]{xy}
 \usepackage{appendix}
 \usepackage{xy}
\usepackage[backref]{hyperref}
\hypersetup{
  colorlinks   = true,          
  urlcolor     = blue,          
  linkcolor    = purple,          
  citecolor   = blue             
}
 \usepackage{cleveref}
\usepackage{tikz-cd}
\usepackage{hyperref}
\usepackage{cleveref}
\usepackage{mathtools}
\usepackage{xcolor}
\usepackage[old]{old-arrows}
\usepackage{qtree,tensor}
\usepackage{comment}
\usepackage{enumerate}
\usepackage[normalem]{ulem}
\def\ddefloop#1{\ifx\ddefloop#1\else\ddef{#1}\expandafter\ddefloop\fi}
\usepackage{xpatch}
\xpatchcmd{\paragraph}{\normalfont}{{\normalfont\bfseries}}{}{}

\def\ddef#1{\expandafter\def\csname bb#1\endcsname{\ensuremath{\mathbb{#1}}}}
\ddefloop ABCDEFGHIJKLMNOPQRSTUVWXYZ\ddefloop

\def\ddef#1{\expandafter\def\csname ff#1\endcsname{\ensuremath{\mathfrak{#1}}}}
\ddefloop ABCDEFGHIJKLMNOPQRSTUVWXYZ\ddefloop

\def\ddef#1{\expandafter\def\csname cc#1\endcsname{\ensuremath{\mathcal{#1}}}}
\ddefloop ABCDEFGHIJKLMNOPQRSTUVWXYZ\ddefloop

\newcommand{\pre}{\mathrm{pre}}
\newcommand{\Spec}{\mathrm{Spec}\,}

\newcommand{\Sch}{\mathrm{Sch}}

\newcommand{\op}{\mathrm{op}}

\newcommand{\Bl}{\mathrm{Bl}}

\newcommand{\reg}{\mathrm{reg}}

\newcommand{\ord}{\mathrm{ord}\,}
\renewcommand{\div}{\mathrm{div}\,} 
\newcommand{\scalar}[1]{\langle #1\rangle}

\newcommand{\Pic}{\mathrm{Pic}}

\newcommand{\p}{x}
\newcommand{\compp}{C_\p}
\newcommand{\vanishp}{V_\p}
\newcommand{\betap}{\beta_\p}
\newcommand{\betapq}{\beta_{\p,\q}}
\newcommand{\coord}{z}
\newcommand{\coordp}{z_\p}
\newcommand{\GIT}{/\!\!/}
\newcommand{\Hom}{\operatorname{Hom}}

\newcommand{\q}{{q}}

\newcommand{\qreg}{{q_\reg}} 
\newcommand{\qone}{{q_1}} 
\newcommand{\qtwo}{{q_2}} 
 
\newcommand{\qonereg}{{q_{1,\reg}}} 
\newcommand{\qtworeg}{{q_{2,\reg}}} 
\newcommand{\qpr}{{q'}} 
\newcommand{\gr}{\mathrm{gr}}
\newcommand{\qext}{{q_\gr}}
\newcommand{\Qpre}{{\ccQ_{g,n}^{\pre}}} 

\newcommand{\ibar}{{\ccQ(\iota)}} 
\newcommand{\imap}{{\overline{\ccM}(\iota)}} 
\newcommand{\Mgn}{{\overline{\ccM}_{g,n}}} 
\newcommand{\Qgn}{{\ccQ_{g,n}}} 
\newcommand{\css}{\overline{\ccM}^c(X)} 
\newcommand{\cssgn}{\overline{\ccM}^c_{g,n}(X,\beta)}
\newcommand{\Qs}{{\ccQ(s)}}

\newcounter{mainresults}
\newtheorem{theorem}[subsubsection]{Theorem}
\newtheorem{corollary}[subsubsection]{Corollary}

\newtheorem{lemma}[subsubsection]{Lemma}
\newtheorem{proposition}[subsubsection]{Proposition}
\newtheorem{maintheorem}[mainresults]{Theorem}

\theoremstyle{definition}
\newtheorem{definition}[subsubsection]{Definition}
\newtheorem{construction}[subsubsection]{Construction}
\newtheorem{mainconstruction}[mainresults]{Construction}
\newtheorem{maindefinition}[mainresults]{Definition}

\newtheorem{remark}[subsubsection]{Remark}
\newtheorem{example}[subsubsection]{Example}

\newtheorem{notation}[subsubsection]{Notation}

\title[The contraction morphism between maps and quasimaps]{The contraction morphism between maps and quasimaps to toric varieties}
\author[Alberto Cobos Rabano]{Alberto Cobos Rabano}
\address{Alberto Cobos Rabano, KU Leuven, Department of Mathematics, Celestijnenlaan 200B box 2400, BE-3001 Leuven, Belgium}
\email{alberto.cobosrabano@kuleuven.be}
\date{}

\begin{document}

\begin{abstract}
    Given $X$ a smooth projective toric variety, we construct a morphism from a closed substack of the moduli space of stable maps to $X$ to the moduli space of quasimaps to $X$. If $X$ is Fano, we show that this morphism is surjective. The construction relies on the notion of degree of a quasimap at a base-point, which we define. We show that a quasimap is determined by its regular extension and the degree of each of its basepoints.
\end{abstract}

\maketitle

\tableofcontents

\section{Introduction}

\subsection{Results}

To a smooth projective toric variety $X$ we can associate two moduli spaces: $\Mgn(X,\beta)$, the moduli of stable maps, and $\Qgn(X,\beta)$, the moduli of stable toric quasimaps. These moduli spaces are two different compactifications of the space of maps from smooth curves to $X$. We study a geometric comparison of these compactifications. Our main result is the following.

\begin{mainconstruction}[{\Cref{constr: c_X}}]\label{intro_constr: c_X}
    Let $X$ be a smooth projective toric variety. We construct a closed substack $\cssgn$ of $\Mgn(X,\beta)$ and a morphism of stacks 
    \[ 
        c_X\colon \cssgn \to \Qgn(X,\beta)
    \]
    extending the identity on the locus of maps from a smooth source.
\end{mainconstruction}

Our motivation comes from the comparison between Gromov--Witten and quasimap invariants of toric Fano varieties in enumerative geometry. 
For $X=\bbP^N$, the morphism $c_{\bbP^N}$, which is defined globally, agrees with the one used in \cite{stable_quotients} to compare the virtual fundamental classes of both spaces. 
More generally, $c_X$ is defined globally on $\Mgn(X,\beta)$ if all toric divisors in $X$ are nef.
\Cref{intro_constr: c_X} paves the way for a geometric comparison for toric varieties in any genus. The following result makes the Fano case easier to treat.

\begin{maintheorem}[\Cref{thm: surjectivity Fano}]\label{intro_thm: surjectivity Fano}
    Let $X$ be a smooth Fano toric variety. The contraction morphism
    \[
        c_X\colon \cssgn\to \ccQ_{g,n}(X,\beta)
    \]
   is surjective.
\end{maintheorem}

\Cref{intro_constr: c_X} relies on the contraction morphism for (products of) projective spaces. One idea would be to use an embedding $\iota\colon X\hookrightarrow \bbP^N$ and the functoriality of maps and quasimaps. This fails because the morphism $\ibar\colon \Qgn(X,\beta) \to \Qgn(\bbP^N,\iota_\ast \beta)$ is not a closed embedding in general, see \Cref{ex: quasimaps do not embed}, which is due to Ciocan-Fontanine and  which first appeared in \cite[Remark 2.3.3]{Battistella_Nabijou}. This situation motivates the following result. Let $A_1(X)$ denote the group of $1$-cycles on $X$ modulo rational equivalence.

\begin{maintheorem}[\Cref{cor: injective on A1 implies closed embedding}]\label{intro_cor: injective on A1 implies closed embedding}
    Let $\iota\colon X\to Y$ be a  closed embedding between smooth projective toric varieties. If $\iota_\ast \colon A_1(X)\to A_1(Y)$ is injective, then the morphism
    \[
        \ibar \colon \Qgn(X,\beta)\to \Qgn(Y,\iota_\ast \beta)
    \]
    is a closed embedding.
\end{maintheorem}

\Cref{intro_cor: injective on A1 implies closed embedding} leads to \Cref{intro_constr: c_X} following the strategy previously sketched. More precisely, we replace the embedding into projective space by an embedding $\iota\colon X\hookrightarrow \bbP$, with $\bbP = \bbP^{N_1}\times \ldots \times \bbP^{N_s}$, such that $\iota_\ast$ is injective on curve classes (we call embeddings with this property \textit{epic}, see \Cref{prop: equivalence epic morphisms}). Such an embedding exists by \Cref{prop: embedding with surjection on Pic}. With this we obtain a diagram
\[
\begin{tikzcd}\Mgn(X,\beta)\arrow[hookrightarrow]{r}{\imap}\arrow[dashed,d,"c_X"]&\Mgn(\bbP^{n_1}\times \ldots \times \bbP^{n_k},\iota_\ast\beta)\arrow[d,"c_{\bbP}"]\\
\Qgn(X,\beta)\arrow[hookrightarrow]{r}{\ibar}&\Qgn(\bbP^{n_1}\times \ldots \times \bbP^{n_k},i_*\beta).
\end{tikzcd}
\]
The dashed vertical arrow is defined on the locus of maps in $\Mgn(X,\beta)$ for which $c_\bbP\circ \imap$ factors through the image of $\ibar$, that is on
\[
    \cssgn \coloneqq \Qgn(X,\beta) \times_{\Qgn(\bbP,\iota_\ast \beta)} \Mgn(\bbP,\iota_\ast \beta).
\]

Most of the paper relies, directly or indirectly, on the notion of degree of a basepoint.

\begin{maindefinition}[\Cref{def: degree of a basepoint}]\label{intro_def: degree of a basepoint}
    Let $\q$ be a quasimap to a smooth projective toric variety $X$ and let $\p$ be a nonsingular point of the source. We construct a class $\betap \in A_1(X)$ called the degree of the quasimap $\q$ at the point $\p$.
\end{maindefinition}

The degree $\betap$ is constructed combinatorially in \Cref{proposition: existence choice of sigma general case} and completely characterized in \Cref{prop: uniqueness of degree of a basepoint}. We give a geometric interpretation, under certain assumptions, in \Cref{rmk: geometric_interpretation_degree}. It
recovers the notion of length of a basepoint in \cite[Definition 7.1.1]{C-FKM} (for toric quasimaps), see \Cref{prop: degree_generalizes_length}. We show in \Cref{lem: equality of quasimaps} that a toric quasimap is determined by its regular extension (\Cref{def: regular extension}) and the degree of each of its basepoints. The degree of a basepoint is essential in \Cref{intro_cor: injective on A1 implies closed embedding} and in the proof of \Cref{intro_thm: surjectivity Fano}.

\subsection{Context}

The moduli space $\Qgn(X,\beta)$ of stable toric quasimaps was introduced in \cite{C-FK}, in relation to \cite{Givental_mirror_theorem, stable_quotients}. These constructions have been further generalized to include quasimaps to more general GIT quotients \cite{C-FKM}, stacks \cite{cheong2015orbifold}, and a stability parameter \cite{ciocan2013wall, toda2011moduli}.

Quasimaps provide a compactification of the space of maps from smooth curves to $X$. The advantage over other compactifications relies in the fact that quasimap invariants, or more precisely the associated $I$-function, are often computable \cite{cooper2014mirror,ciocan2016big,kim2018mirror,ciocan2020quasimap, Battistella_Nabijou}. Quasimaps have also been successfully used to prove results concerning Gromov--Witten invariants \cite{lho2018stable, fan2017mathematical,CRMS}.

The contraction morphism $c_X$ from \Cref{intro_constr: c_X} appeared in \cite{stable_quotients} in the case $X=\bbP^N$ and in \cite{C-FK} for moduli spaces of maps and quasimaps with one parametrized component. A similar picture for Grassmannians appears in \cite{popa_roth, Manolache}, where the authors show that the natural contraction morphism does not extend to the whole moduli space of stable maps. 

The relation between Gromov--Witten invariants and quasimap invariants has been proved in many situations: via localization for projective spaces \cite{stable_quotients} and for complete intersections in projective spaces \cite{clader2017higher} and via wall-crossing \cite{ciocan2017higher, ciocan2020quasimap}. In particular, Gromov-Witten and quasimap invariants of a smooth Fano toric variety agree \cite{ciocan2017higher}. 

Our motivation is to use \Cref{intro_constr: c_X} to prove this statement in a more geometric way as in \cite{manolache2012virtual, Manolache} and at the level of derived structures, extending \cite{kern2022derived}. Besides giving a better geometric understanding of the map-to-quasimap wall-crossing, we hope such a result would contribute to the development of techniques that relate virtual classes of moduli spaces which are ``virtually birational'' but not necessarily birational.

\subsection{Outline of the paper}

We summarize the content of each section, highlighting the main results.

In \Cref{sec: background quasimaps} we review standard content of toric geometry and toric quasimaps while introducing our notations.
 
The main purpose of \Cref{sec: basepoint class} is to define the degree of a basepoint in \Cref{def: degree of a basepoint}. Before that, we construct it combinatorially in \Cref{proposition: existence choice of sigma general case}. The importance of the definition is exhibited by its geometric characterization in \Cref{prop: uniqueness of degree of a basepoint}. We collect basic properties of the degree of a basepoint in \Cref{prop: properties degree of a basepoint} and we show that it recovers the length of a point in \Cref{prop: degree_generalizes_length}.
 
We present the main problem of \Cref{sec: embeddings of toric varieties} in  \Cref{ex: quasimaps do not embed}: given a closed embedding $\iota$ between smooth projective toric varieties, the induced morphism $\ibar$ between quasimap spaces may not be a closed embedding. We identify that the property of $\ibar$ being a monomorphism is related to the induced morphism $\iota_\ast$ on the Chow group $A_1$ being injective. This observation motivates the notion of \textit{epic} morphism, which we discuss in more generality in \Cref{subsec: epic morphisms}. The main result of this subsection is \Cref{thm: toric is projectively epic}: every smooth projective toric variety admits an epic closed embedding into a product of projective spaces. In \Cref{subsec: criterion_qmaps_embedd} we show that if $\iota$ is epic, then $\ibar$ is a closed embedding (\Cref{cor: injective on A1 implies closed embedding}). The proof is as follows: first, we show (\Cref{lem: equality of quasimaps}) that a quasimap is determined by its regular extension (introduced in \Cref{def: regular extension}) and the degree of each of its basepoints. This is used in \Cref{thm: fibres of ibar} to characterize the fibres of $\ibar$, from which \Cref{cor: injective on A1 implies closed embedding} follows. 

In \Cref{sec: contraction}, we 
construct a contraction morphism from a closed substack of the moduli space $\Mgn(X,\beta)$ to $\Qgn(X,\beta)$  (\Cref{constr: c_X}). 
We describe explicitly the contraction morphism and the locus where it is defined in \Cref{prop: description css and c_X}, at the level of closed points.
	
In \Cref{sec: surjectivity Fano}, we show that the contraction morphism is surjective for Fano targets (\Cref{thm: surjectivity Fano}), and more generally for $X$ satisfying the condition in \Cref{rmk: relax Fano in surjectivity}. The proof relies on the notion of degree of a basepoint (\Cref{def: degree of a basepoint}).

\subsection{Acknowledgements}

I am very thankful to Cristina Manolache for posing this question and for the numerous discussions we have had about this topic. I would also like to thank Samuel Johnston, Jingxiang Ma, Etienne Mann, Navid Nabijou, Luis Manuel Navas Vicente and Menelaos Zikidis for helpful discussions and  Navid Nabijou, Dhruv Ranganathan and Evgeny Shinder for their comments on previous drafts.

The content of this paper is based on the author's PhD Thesis \cite[Chapter 2]{Cobos_thesis}. The author was supported by Fonds Wetenschappelijk Onderzoek (FWO) with reference G0B3123N.

\section{Background}\label{sec: background quasimaps}

\subsection{Toric varieties} \label{background: toric varieties}

Let $X$ be a toric variety, that is, a normal variety over $\bbC$ containing a torus $T$ as a dense open subset and such that the natural group action of $T$ extends to an action of $T$ on $X$. The character and co-character lattices of $T_X$ are denoted by $M_X$ and $N_X$ respectively. We will omit the subindex $X$ if there is no risk of confusion.

Given a fan $\Sigma$, we denote by $\Sigma(k)$ and by $\Sigma_{\max}$ the sets of $k$-dimensional cones and of maximal cones in $\Sigma$, respectively. Given a cone $\sigma$, we denote by $\sigma(k)$ the set of its $k$-dimensional faces. Elements in $\Sigma(1)$ are called \textit{rays of $\Sigma$} and will typically be denoted by the letters $\rho$ or $\tau$. To each ray $\rho$ we can associate a divisor $D_\rho^X$ in $X$ and an element $u_\rho \in N$, the unique generator of the semigroup $\rho\cap N$.

Recall that $X$ is proper if and only if $\cup_{\sigma \in \Sigma} \sigma = N\otimes_{\bbZ} \bbR$, and that $X$ is smooth if and only if for each cone $\sigma \in \Sigma$ it holds that the set $\{u_\rho \colon \rho \in \sigma(1)\}$ is part of a $\bbZ$-basis of $N$. The canonical divisor of $X$ can be written as $K_X = -\sum_{\rho \in \Sigma(1)} D_\rho^X$. We say that a smooth toric variety $X$ is Fano if the anticanonical divisor $-K_X$ is ample. 
If $X$ is smooth, Fano and proper, then $-K_X$ is very ample by \cite[Theorem 6.1.15]{CLS}, in particular $X$ is projective. In that case, we refer to the closed embedding $X\to \bbP(H^0(-K_X))$ as the \textit{anticanonical embedding} of $X$.

For a (not necessarily toric) normal variety $X$, we denote by $A_k(X)$ the group of $k$-dimensional cycles on $X$ modulo rational equivalence. If $X$ is smooth of dimension $n$, intersection product induces a perfect pairing
\[
    A_k(X) \times A_{n-k}(X) \to \bbZ.
\]
In that case, we denote $A^k(X) \coloneqq A_{n-k}(X)$. In particular, $A_1(X)$ is dual to $A^1(X)$, which is isomorphic to the group of Cartier divisors modulo principal divisors and also isomorphic to the Picard group $\Pic(X)$ of isomorphisms classes of invertible sheaves on $X$.

Let $X$ be a smooth proper toric variety. The ring $S^X = \bbC[z_\rho]_{\rho\in\Sigma(1)}$ is called the total coordinate ring or the Cox ring of $X$. It has a natural $\Pic(X)$-grading by declaring that $\deg(z_\rho) = [D_\rho^X]\in \Pic(X)$. Such $X$ can be written as a geometric quotient
\begin{equation}\label{equation:GIT_presentation}
    X \simeq   (\bbA^{\Sigma_X(1)} \setminus Z(\Sigma_X))\GIT G_X,
\end{equation}
where $\bbA^{\Sigma_X(1)} = \Spec(S^X)$, where $G_X=\Hom(A^{1}(X),\bbC^*)$ and where $Z(\Sigma_X)$ is the zero set of the ideal
\[
        B(\Sigma_X) = \left\langle z^{\hat{\sigma}} \colon \sigma\in\Sigma_X(\dim X) \right\rangle \subset S^X
\]
with 
    \[
        z^{\hat{\sigma}} = \prod_{\rho\notin\sigma(1)} z_\rho \in S^X
    \]
for each cone $\sigma \in \Sigma_X$. 
A subset $\ccP$ of $\Sigma_X(1)$ is a \textit{primitive collection} if $\ccP$ is not contained in any cone of $\Sigma_X$, but each proper subset of $\ccP$ is contained in some cone of $\Sigma_X$. The irreducible decomposition of $Z(\Sigma_X)$ is
\begin{equation}\label{eq: Z Sigma primitive collections}
    Z(\Sigma_X) = \bigcup_{\ccP} V(z_\rho \colon \rho \in \ccP),
\end{equation}
where the union is over all the primitive collections $\ccP$ in $\Sigma_X(1)$.

\subsection{Morphisms to a smooth toric variety}

The functor of points of a smooth toric variety can be described in terms of line bundle-section pairs due to \cite{Cox_functor}. 

Let $X_\Sigma$ be a smooth toric variety with fan $\Sigma$ in a lattice $N$ and let $M=\operatorname{Hom}_\bbZ(N,\bbZ)$. 

\begin{definition}\label{def: sigma-collection}
    A $\Sigma$-collection on a scheme $S$ is a triplet
    \[
        (L_\rho)_{\rho\in \Sigma(1)}, (s_\rho)_{\rho\in \Sigma(1)}, (c_m)_{m\in M}
    \] 
    of line bundles $L_\rho$ on $S$, sections $s_\rho \in H^0(S,L_\rho)$ and isomorphisms
    \[
        c_m\colon \otimes_{\rho\in\Sigma(1)} L_\rho^{\otimes \langle m,u_\rho\rangle} \simeq \ccO_S 
    \]
    satisfying the following conditions:
    \begin{enumerate}
        \item compatibility: $c_m\otimes c_{m'} = c_{m+m'}$ for all $m,m'\in M$,
        \item non-degeneracy: for each $\p\in S$ there is a maximal cone $\sigma\in\Sigma$ such that $s_\rho(\p)\neq 0$ for all $\rho\not\subset \sigma$.
    \end{enumerate}
    An equivalence between two $\Sigma$-collections on $S$ is a collection of isomorphisms among  the line bundles that preserve the sections and trivializations.
\end{definition}

Consider the functor
\[
    \ccC_\Sigma\colon \Sch^{\op} \to \mathrm{Set}
\]
that associates to each scheme $S$ the collection of $\Sigma$-collections on $S$ up to equivalence, with morphisms naturally defined by pull-back. Then \cite[Thm 1.1]{Cox_functor} states that $\ccC_\Sigma$ is represented by $X_\Sigma$, with universal family the $\Sigma$-collection on $X_\Sigma$ that consists of the line bundles $\ccO_{X_\Sigma}(D_\rho)$ with their natural sections and trivializations $c_m$ given by the characters $\chi^m$ of the dense torus in $X_\Sigma$.

\subsection{Toric quasimaps}

The previous interpretation of the functor of points of a smooth toric variety was used in \cite{C-FK} to define the notion of quasimaps to a toric variety, by relaxing the non-degeneracy condition above.

\begin{definition}(\cite[Definition 3.1.1]{C-FK})\label{def: quasimap}
    Let $X_\Sigma$ be a smooth projective toric variety. A \textit{(prestable toric) quasimap} $\q$ to $X_\Sigma$ consists of
    \begin{enumerate}
        \item a connected nodal projective curve $C$ of genus $g$ with $n$ distinct non-singular markings,
        \item line bundles $L_\rho$ on $C$ for $\rho\in\Sigma(1)$,
        \item sections $s_\rho\in H^0(C,L_\rho)$ for $\rho\in\Sigma(1)$ and 
        \item trivializations $c_m\colon \otimes_{\rho\in\Sigma(1)} L_\rho^{\otimes \langle m,u_\rho\rangle} \simeq \ccO_C$ for $m\in M$
    \end{enumerate}
    satisfying
    \begin{enumerate}
        \item compatibility: $c_m\otimes c_{m'} = c_{m+m'}$ for all $m,m'\in M$,
        \item quasimap non-degeneracy: there is a finite (possible empty) set $B\subseteq C$ of nonsingular points disjoint from the markings on $C$, such that for every $\p\in C\setminus B$ there is a maximal cone $\sigma\in\Sigma$ such that $s_\rho(\p)\neq 0$ for all $\rho\not\subset \sigma$.
    \end{enumerate}
    We will write $\q = (C,L_\rho, s_\rho, c_m)$ or $\q\colon C \dashrightarrow X$ to denote a quasimap $\q$.
\end{definition}

Points in the set $B = B_\q$ in \Cref{def: quasimap} are called \textit{basepoints} of the quasimap. A toric quasimap defines a regular morphism $C\setminus B \to X_\Sigma$. In fact, \cite[Thm 1.1]{Cox_functor} shows that, locally, $\q$ can be lifted to a morphism to $\bbA^{\Sigma_X(1)}$ which factors through $\bbA^{\Sigma_X(1)} \setminus Z(\Sigma_X)$ exactly along $C\setminus B$ (see \Cref{equation:GIT_presentation}).

\begin{definition}\label{def: degree quasimap}
     The degree of a toric quasimap $\q= (C, L_\rho, s_\rho, c_m)$ is the class $\beta=\beta_C\in A_1(X_\Sigma)$ determined by the conditions
    \begin{equation}\label{eq: def degree}
        \beta\cdot D_\rho = \deg(L_\rho)
    \end{equation}
    for all $\rho\in\Sigma(1)$. Similarly, we associate a curve class $\beta_{C'} = \beta_{C',\q}\in A_1(X_\Sigma)$  to every irreducible component $C'$ of $C$ by replacing $\deg(L_\rho)$ with $\deg(L_\rho\mid_{C'})$ in \Cref{eq: def degree}. By \cite[Lemma 3.1.3]{C-FK}, $\beta_{C'}$ is effective.
\end{definition}

If $\q$ has no basepoints, this agrees with the usual notion of degree of a map. However, in the presence of basepoints, the degree of a quasimap does not agree with the degree of its associated regular map by \Cref{prop: properties degree of a basepoint}. See \Cref{ex: extension has different degree}.

\subsection{Moduli space of stable toric quasimaps}

The notion of stability for families of toric quasimaps was introduced in \cite{C-FK}. 

\begin{definition}\label{def: quasimap stability}
    Let $X_\Sigma$ be a smooth projective toric variety and let  $\{\alpha_\rho\}_{\rho\in\Sigma(1)}$ such that $L=\otimes_\rho \ccO_{X_\Sigma}(D_\rho)^{\otimes \alpha_\rho}$ is a very ample line bundle on $X_\Sigma$. Then a toric quasimap $\q=(C, L_\rho, s_\rho,c_m)$ to $X_\Sigma$ is \textit{stable} if the line bundle
    \[
        \omega_C(p_1+\ldots+p_n) \otimes \ccL^\epsilon
    \]
    is ample for all $\epsilon \in \bbQ_{>0}$, where $\ccL = \otimes_\rho L_\rho^{\otimes \alpha_\rho}$ and $p_i$ denote the marked points in $C$.
\end{definition}

\begin{remark}
    The notion of a quasimap being stable is independent of the polarization chosen in \Cref{def: quasimap stability} by \cite[Lemma 3.1.3]{C-FK}. Also, stability imposes the inequality $2g-2+n \geq 0$, which we assume from now on.
\end{remark}

\begin{definition}(\cite[Definition 3.1.7]{C-FK})\label{def: family of quasimaps}
    A \textit{family of genus $g$ stable toric quasimaps to $X_\Sigma$ of class $\beta$ over $S$} is 
    \begin{enumerate}
        \item a flat, projective morphism $\pi\colon\ccC\to S$ of relative dimension one,
        \item sections $p_i\colon S \to \ccC$ of $\pi$ for $1\leq i\leq n$,
        \item line bundles $L_\rho$ on $\ccC$ for $\rho\in\Sigma(1)$,
        \item sections $s_\rho\in H^0(\ccC,L_\rho)$ for $\rho \in \Sigma(1)$ and 
        \item compatible trivializations $c_m\colon \otimes_{\rho\in\Sigma(1)} L_\rho^{\otimes \langle m,u_\rho\rangle} \simeq \ccO_\ccC$ for $m\in M$
    \end{enumerate}
    such that the restriction of the data to every geometric fibre of $\pi$ is a stable toric quasimap of genus $g$ and class $\beta$.
\end{definition}

\begin{theorem}(\cite[Theorems 3.2.1, 4.0.1]{C-FK})
    The moduli space $\Qgn(X_\Sigma,\beta)$ of stable quasimaps of degree $\beta$ to a smooth projective toric variety $X_\Sigma$ from genus-$g$ $n$-marked curves is a proper Deligne--Mumford stack, of finite type over $\bbC$.
\end{theorem}

The moduli space $\Mgn(X_\Sigma,\beta)$ of stable maps of degree $\beta$ from genus-$g$ $n$-marked curves to a smooth projective toric variety $X_\Sigma$ can be recovered similarly. We define a \textit{(prestable) map} to $X_\Sigma$ to be a prestable toric quasimap (\Cref{def: quasimap}) whose set of basepoints $B$ is empty. Then, we impose the following stability condition for maps (which is different from \Cref{def: quasimap stability}):
\[
        \omega_C(p_1+\ldots+p_n) \otimes \ccL^2
 \]
is ample. As a result,  every map is a quasimap, but stability need not be preserved. Note that both stability conditions are defined to ensure that the objects have finitely-many automorphisms, so that $\Qgn(X_\Sigma,\beta)$ and $\Mgn(X_\Sigma,\beta)$ are Deligne-Mumford stacks.

\subsection{Functoriality of quasimaps}\label{subsec: functoriality quasimaps}

We fix a (not necessarily toric) closed embedding $\iota\colon X\to Y$ between smooth projective toric varieties for the rest of \Cref{subsec: functoriality quasimaps}. The goal of this subsection is to describe the morphism
\[
    \ibar: \Qgn(X,\beta)\to \Qgn(Y,\iota_\ast \beta) 
\]
associated to $\iota$. We follow \cite[Appendix B]{Battistella_Nabijou_v1}.\\

We denote rays in the fan $\Sigma_X$ of $X$ by $\rho$ and rays in the fan $\Sigma_Y$ of $Y$ by $\tau$. 
By \cite[Theorem 3.2]{Cox_functor}, the morphism $\iota$ corresponds to a collection of homogeneous polynomials $P_\tau \in S^X$ for $\tau \in \Sigma_Y(1)$ of degree $d_\tau \coloneqq \iota^\ast \ccO_Y(D^Y_\tau) \in \Pic(X)$ satisfying the following conditions:
\begin{enumerate}
    \item $\sum_{\tau \in \Sigma_Y(1)} \beta_\tau \otimes u_\tau = 0$ in $\Pic(Y)\otimes N$.
    \item $(P_\tau(x_\rho))\notin Z(\Sigma_Y)$ in $\bbA^{\Sigma_Y(1)}$ for every $(x_\rho)\notin Z(\Sigma_X)$ in $\bbA^{\Sigma_X(1)}$.
\end{enumerate}
The relation between $\iota$ and $(P_\tau)$ is that the morphism 
\[
    \tilde{\iota}\colon \bbA^{\Sigma_X(1)}\setminus Z(\Sigma_X) \to \bbA^{\Sigma_Y(1)}\setminus Z(\Sigma_Y) 
\]
defined by $\tilde{\iota}(x_\rho) \coloneqq (P_\tau(x_\rho))$ is a lift of $\iota$.

For $\tau \in \Sigma_Y(1)$, we write the polynomial $P_\tau$ as follows:
\begin{equation}\label{eq: expression_polys}
    P_\tau (x_\rho) = 
    \sum_{\underline{a}} P_\tau^{\underline{a}} (t_\rho) =
    \sum_{\underline{a}} \mu_{\underline{a}} \prod_{\rho\in \Sigma_X(1)} x_\rho^{a_\rho},
\end{equation}
where the sum is over a finite number of indices $\underline{a} = (a_\rho) \in \bbN^{\Sigma_X(1)}$ and the coefficients $\mu_{\underline{a}}$ are non-zero. We choose a multi-index $\underline{a}$ apprearing in \Cref{eq: expression_polys} and denote it by $\underline{a}^\tau$.

\begin{construction}\label{constr: functoriality_qmaps}
The morphism  
\[
    \ibar: \Qgn(X,\beta)\to \Qgn(Y,\iota_\ast \beta) 
\]
associates to a family of quasimaps 
\[
    (C, L_\rho, s_\rho, c_{m_X})
\]
in $\Qgn(X,\beta)$ the quasimap 
\[
    (C, L'_\tau, s'_\tau, c'_{m_Y})
\]
described as follows:
\begin{equation}\label{eq: functoriality_qmaps}
    L'_\tau = \bigotimes_{\rho\in\Sigma_X(1)} L_\rho^{\otimes a_{\rho}^\tau}, \quad
    s'_\tau = \mu_{\underline{a}^\tau} \prod_{\rho\in\Sigma_X(1)} s_\rho^{a_\rho^\tau}, \quad
    c'_{m_Y} = c_{m_X},
\end{equation}
where $m_X \in M_X$ is uniquely determined by the conditions
\[
    \scalar{m_X,u_\rho} = \sum_{\tau\in \Sigma_Y(1)} a_\rho^\tau \scalar{m_Y,u_\tau}
\]
for every $\rho\in\Sigma_X(1)$
\end{construction}

\begin{remark}\label{rmk: no_stabilization}
    In \Cref{constr: functoriality_qmaps}, the underlying curve $C$ is the not altered because $\iota$ is a closed embedding, thus no stabilization is required. The general case, where $\iota$ is not a closed embedding, is discussed in \cite[Appendix B]{Battistella_Nabijou_v1}.
\end{remark}

\begin{remark}\label{rmk: functoriality_same_maps_qmaps}
    Let $\Qpre(X,\beta)$ denote the moduli stack of $n$-marked genus-$g$  (prestable toric) quasimaps to $X$ of class $\beta$. The description of $\ibar$ in \Cref{constr: functoriality_qmaps} in terms of line bundle-section pairs is the same as that of the morphisms 
    \[
        \ibar \colon \Qpre(X,\beta)\to \Qpre(Y,\iota_\ast \beta)
    \]
    and
    \[
        \imap \colon \Mgn(X,\beta)\to \Mgn(Y,\iota_\ast \beta).
    \]
\end{remark}

\begin{remark}\label{rmk: coefficientes_a_explained}
    The coefficients $a_{\rho}^\tau$ used in \Cref{constr: functoriality_qmaps} satisfy the relations
    \[
        \iota^{\ast} [D^Y_\tau] = \sum_{\rho\in\Sigma_X(1)} a_\rho^\tau [D^X_\rho]
    \]
    for all $\tau\in\Sigma_Y(1)$.
\end{remark}

\subsection{Contraction morphism for projective space}\label{subsec: contraction for Pn}

Fix $N\geq 1$ and let $X_\Sigma = \bbP^N$.
There is a natural morphism 
\[
    c_{\bbP^N}\colon \overline{\ccM}_{g,n}(\bbP^N,\beta) \to \Qgn(\bbP^N,\beta),
\]
called the contraction or comparison morphism for $\bbP^N$. It can be described as follows: a stable map is naturally a quasimap, but it is stable as a quasimap if and only if there are no rational tails. A rational tail $T$ is a tree of rational components with no marked points and such that $T\cap \overline{(C\setminus T)}$ is a singleton (necessarily a node inside $C$). Let $T$ be a rational tail of a stable map $f\colon C\to \bbP^N$, let $\hat{C} = \overline{C\setminus T}$ and let $p = T\cap \hat{C}$.
One must contract $T$ and define $\hat{L}_\rho = L_\rho\mid_{\hat{C}} \otimes \ccO_{\hat{C}} (d_{T,\rho} p)$ with $d_{T,\rho} = \deg(L_\rho\mid_T)$. Similarly, one must take $\hat{s}_\rho = s_\rho \otimes \coordp^{d_{T,\rho}}$, with $\coordp$ a local parameter at $\p$ in $\hat{C}$. Doing this for every rational tail $T$ of $f$ produces the stable quasimap $c_{\bbP^N}(f)$. The description of $c_{\bbP^N}$ for families can be found in \cite[Theorem 7.1]{popa_roth}.

\section{The degree of a basepoint}\label{sec: basepoint class}

We fix a smooth proper toric variety $X$ with fan $\Sigma$ for the rest of \Cref{sec: basepoint class}.

Every toric quasimap $\q\colon C\dashrightarrow X$ has a degree $\beta = \beta_\q$ by \Cref{def: degree quasimap}, which is an effective curve class in $A_1(X)$. Moreover, $\q$ admits an associated regular morphism $\qreg\colon C \to X_\Sigma$ (\Cref{def: regular extension}), which also has a degree $\beta_{\qreg}$. However, if $\q$ has basepoints, then $\beta\neq\beta_{\qreg}$, see \Cref{ex: extension has different degree}.

In \Cref{def: degree of a basepoint}, we attach an effective curve class $\betap$ to each point $\p$ of $C$, called the degree of $\q$ at $x$. We show in \Cref{prop: properties degree of a basepoint} that the degree $\betap$ is 0 unless $\p$ is a basepoint, and that, as $\p$ varies over all basepoints, the degrees $\betap$ explain the discrepancy between $\beta$ and $\beta_{\qreg}$, since
\[
    \beta-\beta_{\qreg} = \sum_{\p\in B} \betap.
\]
We show in \Cref{prop: degree_generalizes_length} that the degree $\betap$ recovers the length of $\q$ at $\p$, introduced in \cite[Def. 7.1.1]{C-FKM}, for toric quasimaps.\\

\subsection{The regular extension of a quasimap}

\begin{definition}\label{def: regular extension}
    Let $\q$ be a quasimap to $X$. Then $\q$ defines a regular morphism
    \[
        \q\colon C\setminus B\to X.
    \]
    Since $X$ is proper and $B$ consists of smooth points in $C$, there is a unique regular morphism
    \[
        \qreg\colon C\to X
    \]
    extending $\q$. We call $\qreg$ the \textit{(regular) extension} of $\q$.
\end{definition}

\begin{example}\label{ex: extension has different degree}
    In general, $\q$ and $\qreg$ may have different degrees.
    Consider the quasimap 
    \[
        \q \colon \bbP^1 \to \bbP^2\colon [x\colon y]\mapsto [0\colon 0 \colon x]
    \]
    which has a unique basepoint, at the point $[0\colon 1]$. Its extension is the constant map
    \[
        \qreg \colon \bbP^1 \to \bbP^2\colon [x \colon y]\mapsto [0\colon 0 \colon 1].
    \]
    Note that $\q$ has degree 1 but $\qreg$ has degree 0. In particular, if we add two distinct marks to $\bbP^1$, distinct from $[0\colon 1]$, then $\q$ is a stable quasimap but $\qreg$ is not a stable map.
\end{example}

\subsection{Definition of the degree of a basepoint}\label{subsec: degree_basepoint}

As seen in \Cref{ex: extension has different degree}, the degree of a quasimap may be different from the degree of its regular extension. We aim to assign a degree to each basepoint in order to explain this discrepancy. Such degree will be constructed using the vanishing order of the sections at a given basepoint.\\

Given a 
nodal curve $C$, a line bundle $L$ on $C$, a non-zero section $s\in H^0(C,L)$ and a nonsingular point $\p\in C$, we denote by $\ord_\p(s)$ the vanishing order of $s$ at $\p$. 
Furthermore, we declare that $\ord_\p(0) = \infty$ and we extend the natural operation and order on the monoid $\bbZ_{\geq 0}$ to $\bbZ_{\geq 0}\cup \{\infty\}$ in the standard way. This means that we declare $a + \infty = \infty$ and $a<\infty$ for all $a\in\bbZ_{\geq 0}$. It follows from this definition that $\infty - a = \infty$ for every $a\in\bbZ_{\geq 0}$.

We will use the following basic result on toric geometry, whose proof we include for completeness.

\begin{lemma}\label{lem: generators Pic}
    Let $\sigma \in \Sigma(\dim X)$. Then the classes $\{ [D_\rho] \colon \rho \notin \sigma(1) \}$ form a $\bbZ$-basis of $\Pic(X)$. 
\end{lemma}

\begin{proof}
    The Picard group is generated by the classes $[D_\rho]$ for $\rho \in \Sigma(1)$, see \cite[Thm 4.2.1]{CLS}. Let $n=\dim X$ and let $\sigma(1)=\{\rho_1,\ldots,\rho_n\}$. Then the ray generators $\{u_{\rho_1},\ldots,u_{\rho_n}\}$ form a basis of the co-character lattice $N$ by smoothness of $\Sigma$. Let $\{m_1, \ldots, m_n\}$ be its dual basis in the character lattice $M$. Then
	\begin{equation}\label{eq: divisor of a character with dual basis}
		\div(\chi^{m_i}) = \sum_{\rho\in\Sigma(1)} \langle m_i, u_\rho\rangle D_\rho = D_{\rho_i} + \sum_{\rho\notin\sigma(1)} \langle m_i, u_\rho\rangle D_\rho.
	\end{equation}
    where $\chi^m$ is the character corresponding to $m\in M$. Taking linear equivalence, it follows that for $i=1,\ldots, n$ the class $[D_{\rho_i}]$ lies in the subgroup generated by $\{ [D_\rho] \colon \rho \notin \sigma(1) \}$. We conclude since $\Pic(X)$ is free of rank $\lvert \Sigma(1)\rvert - n$ by smoothness and properness of $X$.
\end{proof}

\begin{notation}\label{not: compp_vanishp}
    Let $\q=(C,L_\rho,s_\rho,c_m)$ be a quasimap to $X$ and let $\p\in C$ be a smooth point.  
    \begin{itemize}
        \item We denote by $\compp$ the irreducible component of $C$ containing $\p$ and 
        \item by $\vanishp$ the set of rays $\rho\in\Sigma(1)$ such that $s_{\rho}$ vanishes identically on $\compp$.
    \end{itemize}
\end{notation}

\begin{construction}\label{constr: degree of basepoint no vanishing}
    Let $\q=(C,L_\rho,s_\rho,c_m)$ be a quasimap to $X$ and let $\p\in C$ be a smooth point. 
    Given $\sigma\in\Sigma(\dim X)$ such that $\vanishp \subseteq \sigma(1)$, we define the curve class $\beta(\p,\sigma)$ by the conditions
    \begin{equation}\label{equation:definition_beta_p_sigma}
        \beta(\p,\sigma)\cdot [D_\rho] = \ord_\p(s_\rho\mid_{\compp})   \ \forall \rho\notin\sigma(1).
    \end{equation}
    The class $\beta(\p,\sigma)$ is well-defined due to \Cref{lem: generators Pic}.
\end{construction}

Similarly, to any cone $\sigma\in \Sigma(\dim X)$ and any $a = (a_\rho)\in \bbR^{\Sigma(1)}$ we can associate the unique curve class $\beta(a,\sigma)\in A_1(X)$ satisfying the conditions
\begin{equation}\label{eq: definition beta a sigma}
        \beta(a,\sigma)\cdot [D_\rho] = a_\rho  \ \forall \rho\notin\sigma(1).
\end{equation}

Before we use \Cref{constr: degree of basepoint no vanishing}, we need to prove an elementary lemma about the classes $\beta (a,\sigma)$.

\begin{lemma}\label{lem: technical lemma inequalities dual basis}
    Let $n = \dim X$, let $\sigma\in\Sigma(n)$ be a cone with rays $\sigma(1) = \{\rho_1,\ldots, \rho_n\}$ and let $\{m_1,\ldots,m_n\}$ be the dual basis to the ray generators $\{u_{\rho_1},\ldots, u_{\rho_n}\}$. Let $a = (a_\rho)\in \bbR^{\Sigma(1)}$ and define $u = \sum_{\rho\in\Sigma(1)} a_\rho u_\rho$. Then for $i\in \{1,\ldots, n\}$ we have that
    \begin{equation}\label{eq: equivalent inequalities}
	   a_{\rho_i} \geq \beta(a,\sigma) \cdot [D_{\rho_i}] \iff \scalar{m_i,u}\geq 0.
    \end{equation}
    In particular, $a_{\rho_i} \geq \beta(a,\sigma) \cdot [D_{\rho_i}]$ for all $i\in \{1,\ldots, n\}$ if and only if $u\in \sigma$.
\end{lemma}

\begin{proof}
    With the notations of the statement, it is clear that 
    \[
        u\in \sigma \iff \scalar{m_i,u}\geq 0\ \forall\, i \in \{1,\ldots, n\}.
    \]
    Therefore, it is enough to prove
    the equivalence in \eqref{eq: equivalent inequalities}.
    
    Fix $i \in \{1,\ldots, n\}$. By \Cref{eq: divisor of a character with dual basis,eq: definition beta a sigma}, we have that
	\begin{equation}\label{equation: technical lemma aux one}
	    \beta(a,\sigma)\cdot [D_{\rho_i}] = 
       -\sum_{\rho\notin\sigma(1)} \scalar{m_i, u_\rho} \beta(a,\sigma)\cdot [D_\rho]= 
       -\sum_{\rho\notin\sigma(1)} a_\rho \scalar{m_i, u_\rho}.
	\end{equation}
    On the other hand,
	\begin{equation}\label{equation: technical lemma aux two}
	\scalar{m_i,u} = \scalar{m_i,\sum_{\rho\in\Sigma(1)} a_\rho u_\rho} = a_{\rho_i} + \sum_{\rho \notin \sigma(1)} a_\rho \scalar{m_i,u_\rho}.
	\end{equation}
    Using \Cref{equation: technical lemma aux one,equation: technical lemma aux two}, we conclude that
    \[
        a_{\rho_i} \geq \beta(a,\sigma) \cdot [D_{\rho_i}] \iff a_{\rho_i} + \sum_{\rho \notin \sigma(1)} a_\rho \scalar{m_i,u_\rho} \geq 0 \iff \scalar{m_i,u}\geq 0.\qedhere
    \]
\end{proof}

\begin{proposition}\label{proposition: existence choice of sigma no vanishing}
    Let $\q=(C,L_\rho,s_\rho,c_m)$ be a quasimap to $X$ and let $\p\in C$ be a smooth point such that $\vanishp=\emptyset$. Then there is a cone $\sigma \in \Sigma(\dim X)$ and a curve class $\betap \in A_1(X)$ satisfying
    \begin{enumerate}
        \item\label{item:no_vanishing_equalities} $\ord_\p(s_\rho\mid_{\compp}) = \betap\cdot [D_\rho] \text{ for all } \rho \notin \sigma(1)$,
        \item\label{item:no_vanishing_inequalities} $\ord_\p(s_\rho\mid_{\compp}) \geq \betap\cdot [D_\rho] \text{ for all } \rho \in \Sigma(1)$.
    \end{enumerate}
\end{proposition}

\begin{proof}
    It suffices to prove that there exists $\sigma \in \Sigma(\dim X)$ such that \Cref{item:no_vanishing_inequalities} holds for $\betap \coloneqq \beta(\p,\sigma)$ and all $\rho \notin \sigma(1)$. For that, let $u = \sum_{\rho\in\Sigma(1)} \ord_\p(s_\rho\mid_{\compp}) u_\rho$ and let $\sigma \in \Sigma$ be such that $u\in \sigma$, which exists because $X$ is proper. Then the results follows from \Cref{lem: technical lemma inequalities dual basis} applied to $(a_\rho) = (\ord_{\p}(s_\rho\mid_{\compp}))$.
\end{proof}

Next, we want to generalize \Cref{proposition: existence choice of sigma no vanishing} to the case where $\vanishp$ might be empty.

\begin{remark}
    Given a toric quasimap $\q$ with underlying curve $C$ and a smooth point $\p\in C$, there exists a maximal cone $\sigma$ such that $\vanishp\subseteq \sigma(1)$ by generic non-degeneracy. 
\end{remark}

\begin{proposition}\label{proposition: existence choice of sigma general case}
Let $\q=(C,L_\rho,s_\rho,c_m)$ be a quasimap to $X$ and let $\p\in C$ be a smooth point. There is a cone $\sigma \in \Sigma(\dim X)$ and a curve class $\betap \in A_1(X)$ satisfying
    \begin{enumerate}
        \item\label{item:general_equalities} $\ord_\p(s_\rho\mid_{\compp}) = \betap\cdot [D_\rho] \text{ for all } \rho \notin \sigma(1)$,
        \item\label{item:general_inequalities} $\ord_\p(s_\rho\mid_{\compp}) \geq \betap\cdot [D_\rho] \text{ for all } \rho \in \Sigma(1)$.
        \item\label{item:general_containment} $\vanishp \subseteq \sigma(1)$.
    \end{enumerate}
\end{proposition}

\begin{proof}
        It is clear that we are looking for a curve class of the form $\betap = \beta(a,\sigma)$ for certain integers $(a_\rho)_{\rho\in\Sigma(1)}$ and some $\sigma \in \Sigma(\dim X)$ with the extra conditions: 
        \begin{enumerate}[(i)]
            \item $a_\rho = \ord_\p(s_\rho\mid_{\compp})$ for every $\rho \notin \vanishp$, which is equivalent to \Cref{item:general_equalities}.
            \item $u=\sum_{\rho \in \Sigma(1)} a_\rho u_\rho \in \sigma$, which is equivalent to \Cref{item:general_inequalities} by \Cref{lem: technical lemma inequalities dual basis}.
        \end{enumerate}

        If $\vanishp = \emptyset$ it is enough to let $a_\rho = \ord_\p(s_\rho\mid_{\compp})$ for all $\rho\in \Sigma(1)$, as we did in \Cref{proposition: existence choice of sigma no vanishing}.
        
        In general, let $\tau$ be the cone generated by the rays in $\vanishp$, which lies in $\Sigma$ by smoothness of $X$. Let $N_\tau$ be the sublattice of $N$ generated by $\tau\cap N$ and let $N(\tau) = N/N_\tau$. Let 
        \[
            u_0 = \sum_{\rho\notin\vanishp} \ord_{\p}(s_\rho\mid_{\compp}) u_\rho
        \]
        and let $[u_0]$ denote its class in $N(\tau)$. Since $X_\Sigma$ is complete, so is the closure $V(\tau)$ of the orbit corresponding to $\tau$, which is a toric variety with fan the star of $\tau$. It follows that there is a cone $\sigma \in \Sigma(\dim X)$ which contains $\tau$ as a face and such that $[u_0]\in [\sigma]$ in $N(\tau)$. It is enough to choose $(a_\rho)_{\rho\in\vanishp}$ in such a way that $u\in N$ is a lift of $[u_0]\in N(\tau)$ satisfying $u\in \sigma$.
\end{proof}

Next we show that any two classes satisfying \Cref{proposition: existence choice of sigma general case} are equal, independently of the choice of maximal cone $\sigma\in\Sigma$ used to define it.

\begin{remark}\label{rmk: compatible isos}
    Given a curve class $\beta\in A_1(X)$, we have that
    \[
        \sum_{\rho\in\Sigma(1)} \beta\cdot [D_\rho]\ u_\rho = 0.
    \]
    Therefore, for each smooth point $\p\in C$ there is a natural isomorphism 
    \[
        \psi_{m,\p,\beta}\colon \ccO_C(\sum_{\rho\in\Sigma(1)}\ \beta\cdot [D_\rho]\ \scalar{m,u_\rho}\ \p) = \ccO_C(\scalar{m,\sum_{\rho\in\Sigma(1)}\ \beta\cdot [D_\rho]\ u_\rho}\ \p)
        \to \ccO_C.
    \]
    For fixed $\p$ and $\beta$, these isomorphisms satisfy the compatibility 
    \[
        \psi_{m+m',\p,\beta} = \psi_{m,\p,\beta}\otimes \psi_{m',\p,\beta}.
    \]
\end{remark}

\begin{proposition}\label{prop: uniqueness of degree of a basepoint}
    Let $\q = (C, L_\rho, s_\rho, c_m)$ be a quasimap to $X$, let $\p\in C$ be a smooth point and let $\coord$ be a local coordinate at $\p$. There is a unique curve class $\betap\in A_1(X)$ such that the following data defines a quasimap to $X$ of which $\p$ is not a basepoint:
    \[
        \qpr = (C,L'_\rho, s'_\rho, c'_m)
    \]
    with 
    \begin{align*}
        L'_\rho &\coloneqq L_\rho\otimes \ccO_C(-\betap\cdot [D_\rho]\ \p),\\
        s'_\rho &\coloneqq s_\rho\otimes \coord^{-\betap\cdot [D_\rho]},\\
        c'_m &= c_m\otimes \psi_{-m,\p,\betap},
    \end{align*}
    where $\psi_{m,\p,\betap}$ denotes the isomorphism constructed in \Cref{rmk: compatible isos}
\end{proposition}

\begin{proof}
    Firstly, for any curve class $\betap$ we have that
    \begin{align*}
        \bigotimes_\rho {L'}_\rho^{\otimes \scalar{m,u_\rho}}  = 
        \left(\bigotimes_\rho L_\rho ^{\otimes \scalar{m,u_\rho}}\right)\otimes \ccO_C(-\sum_\rho (\betap\cdot [D_\rho])\  \scalar{m,u_\rho}\ \p).
    \end{align*}
    The first term is isomorphic to $\ccO_C$ via $c_m$ and the second one via $\psi_{-m,\p,\betap}$.

    Secondly, the function $s'_\rho$ is a regular section of the line bundle $L'_\rho$ if and only if 
    \begin{equation}\label{eq: inequality order betap}
        \ord_\p(s_\rho\mid_{\compp}) \geq \betap \cdot [D_\rho] \ \forall \rho\in\Sigma(1).
    \end{equation}
    Therefore, $\qpr$ is a quasimap if and only if $\betap$ satisfies \Cref{item:no_vanishing_inequalities} in \Cref{proposition: existence choice of sigma general case}.

    Similarly to \Cref{not: compp_vanishp}, let $\vanishp'$ be the set of rays $\rho$ such that $s_{\rho}'$ vanishes identically on $\compp$. Note that $\vanishp = \vanishp'$. One can check that that $\p\notin B_{\qpr}$ if and only if there exists a maximal cone $\sigma\in\Sigma$ such that the collection of rays $\rho\in\Sigma(1)$ satisfying $s'_{\rho}(\p) = 0$ is contained in $\sigma$. 
    This happens if and only if $\betap$ satisfies \Cref{item:general_equalities,item:general_containment} in \Cref{proposition: existence choice of sigma general case}. Therefore, existance of $\betap$ follows from \Cref{proposition: existence choice of sigma general case}.
    
    To conclude, we show that $\betap$ is unique. 
    By the previous discussion, it suffices to check that if we have two cones $\sigma,\sigma'\in\Sigma(\dim X)$ with $\vanishp\subseteq \sigma(1)$ and $\vanishp\subseteq \sigma'(1)$, then the curve classes $\beta \coloneqq \beta(\p,\sigma)$ and $\beta'\coloneqq\beta(\p,\sigma')$ defined in \Cref{constr: degree of basepoint no vanishing} agree. 
    By construction, we have that
    \begin{align*}
		\ord_\p(s_\rho\mid_{\compp}) &\geq \beta'\cdot [D_\rho]  \text{ for all } \rho \in \Sigma(1) \text{ and}\\
		\ord_\p(s_\rho\mid_{\compp}) &= \beta\cdot [D_\rho]  \text{ for all } \rho \notin \sigma(1).
    \end{align*}
    It follows that
    \[
		(\beta-\beta')\cdot D_\rho \geq \ord_\p(s_\rho\mid_{\compp})-\ord_\p(s_\rho\mid_{\compp}) = 0 
	\]
    for all $\rho\notin\sigma(1)$. This uses that fact that, since $\vanishp\subseteq \sigma(1)$, we have that $\ord_\p(s_\rho\mid_{\compp})$ is finite for all $\rho\notin\sigma(1)$. 

    Since every ample line bundle $L$ on $X_\Sigma$ can be written with non-negative coefficients in the basis $\{[D_\rho]\colon \rho\notin\sigma(1)\}$ of $\Pic(X_{\Sigma})$, we see that $\beta-\beta'$ is non-negative on the ample cone, therefore also on the nef cone. This means that $\beta-\beta'$ lies in the dual of the nef cone, so it is effective. The same applies to $\beta'-\beta$, so we must have $\beta = \beta'$.
\end{proof}

Note that $\qpr$ might not be stable quasimap even if $\q$ is, see \Cref{ex: extension has different degree}.

\begin{definition}\label{def: degree of a basepoint}
   In the setting of \Cref{prop: uniqueness of degree of a basepoint}, we say that the class $\betap = \betapq \in A_1(X)$ is the \textit{degree of the quasimap $\q$ at the point $\p$}. For a nodal point $\p$, we define $\betap = 0$.
\end{definition}

In \Cref{rmk: geometric_interpretation_degree} we give a more geometric reformulation of $\betap$ in terms of the contraction morphism under certain assumptions. The importance of this notion will become clear in \Cref{subsec: properties_degree_basepoint}. 
we came up with \Cref{def: degree of a basepoint} while studying the non-injectivity of the morphisms $\ibar$ associated to a closed embedding (see \Cref{constr: functoriality_qmaps}), in particular while studying \Cref{ex: quasimaps do not embed}, which we highly recommend.

\subsection{Examples of the degree of a basepoint}\label{subsec: examples degree basepoint}

In \Cref{subsec: degree_basepoint} we have assigned a curve class $\betap\in A_1(X)$ to each smooth point $\p$ of a quasimap to $X$ uniquely determined by \Cref{prop: uniqueness of degree of a basepoint}. The class $\betap$ is of the form $\beta(\p,\sigma)$ (see \Cref{constr: degree of basepoint no vanishing}) for some cone $\sigma \in \Sigma(\dim X)$. Under the assumption that $\vanishp = \emptyset$ (that is, if no section $s_\rho$ vanishes identically on the component $\compp$ containing $\p$), the choice of $\sigma$  (explained in \Cref{proposition: existence choice of sigma no vanishing}) is as follows: it suffices to take $\sigma$ to be any cone in $\Sigma(\dim X)$ containing the element 
\begin{equation}\label{eq: def u}
    u \coloneqq \sum_{\rho\in\Sigma(1)} \ord_\p (s_\rho\mid_{\compp}) u_\rho \in N.
\end{equation}
In general (see \Cref{proposition: existence choice of sigma general case}), if $s_\rho$ vanishes identically on $\compp$, it suffices to replace $\ord_\p (s_\rho\mid_{\compp})$ in \Cref{eq: def u} by a large enough integer $\lambda_\rho \in \bbZ_{\geq 0}$. Intuitively, this corresponds to taking $\ord_\p (s_\rho\mid_{\compp}) = \infty$. In practice, it has the effect of ensuring that $\vanishp \subseteq \sigma(1)$.

\begin{example}[Basepoints of quasimaps to $\bbP^N$]\label{ex: basepoints_PN}
Consider the particular case $X=\bbP^N$. We fix the isomorphism $A_1(\bbP^N) \simeq \bbZ$ mapping the class $\ell$ of a line to $1$.

Let $\q = (C,L_\rho, s_\rho, c_m)$ be a quasimap to $\bbP^N$. Using the isomorphisms $c_m$, we can view $q$ as the choice of a single line bundle $L$ on $C$ with a tuple of sections $s_0,\ldots, s_N \in H^0(C,L)$. The degree of $\q$ is then $\deg(L)$.

The above data determines a rational morphism
\begin{align*}
    C  &\to \bbP^N\\
    x &\mapsto [s_0(x)\colon \ldots \colon s_N(x)],
\end{align*}
because non-degeneracy ensures that there is no irreducible component $C'$ of $C$ where all the restrictions $s_i\mid_{C'}$ are identically 0. A point $x\in C$ is a basepoint if and only if $s_0(x) = \ldots = s_N(x) = 0$; that is, if and only if $x$ lies in the indeterminacy locus of the rational map. 

For a smooth point $\p \in C$, we claim that $\betap$ equals the curve class 
\begin{equation} \label{eq: ex definition dp}
    d_\p \coloneqq \min_{i=0}^N \{\ord_\p (s_i)\} \in A_1(\bbP^N)
\end{equation}
(with the convention that $\ord_\p(s_i) = \infty$ if $s_i$ is identically 0 on the component $\compp$ of $C$ containing $\p$) satisfies the uniqueness property in \Cref{prop: uniqueness of degree of a basepoint}.

Indeed, we can write each section $s_i$ as $s_i=\coordp^{d_\p} s'_i$ with $\coordp$ a local coordinate at $\p$ and $s'_i\in H^0(C,L\otimes \ccO_C(-d_\p \p))$. Furthermore, if the minimum in \Cref{eq: ex definition dp} is achieved at $s_i$, then
\[
    \ord_\p (s'_i) = \ord_\p (s_i) - d_\p = 0,
\]
so $s'_i(\p)\neq 0$. This ensures that $\p$ is not a basepoint of the quasimap $(C,L',s'_0,\ldots, s'_N)$ with $L' = L\otimes \ccO_C(-d_\p \p)$; thus proving the claim.

Note that if we repeat the above for $d'_\p< d_\p$, then the sections $s'_i$ will still vanish simultaneously at $\p$, so $\p$ will remain a basepoint. Similarly, if we choose $d'_\p> d_\p$, then, for $i$ such that the minimum in \Cref{eq: ex definition dp} is achieved at $s_i$, the section $s'_i$ is not a regular section of $L'$. Therefore, the class $d_\p$ is clearly unique in this case.

To conclude, we show that our construction of $\betap$ in \Cref{subsec: degree_basepoint} recovers $d_\p$. 
The fan $\Sigma$ of $\bbP^N$ has $N+1$ rays $\rho_0,\ldots, \rho_N$ and $N+1$ maximal cones, that we label as
\[
    \sigma_{i} = \mathrm{Cone}(\rho_0,\ldots, \widehat{\rho_i},\ldots, \rho_N).
\] 
See \Cref{fig: fan P2} for the fan of $\bbP^2$. Remember that the class of each boundary divisor $D_{\rho_i}$ in $\Pic(\bbP^N)$ equals the class of a hyperplane.

Suppose, for simplicity, that $\vanishp = \emptyset$. Then the class $\beta(x,\sigma_{i}) \in A_1(\bbP^N)$ (see \Cref{constr: degree of basepoint no vanishing}) corresponds to the degree $d_{x,\sigma_{i}} \in \bbZ$ given by
\[
    d_{x,\sigma_{i}} = \beta(x,\sigma_{i}) \cdot [D_{\rho_i}] = \ord_{\p}(s_i).
\]
Therefore, it only remains to check that
\[
    u \coloneqq \sum_{j=0}^N \ord_\p (s_j) u_{\rho_j} \in \sigma_{i} \iff \ord_\p (s_i) = \min_{j=0}^N \{\ord_\p (s_j)\}.
\]
To simplify the notation, let $a_j = \ord_\p (s_j)$. Using that $\sum_{j=0}^N u_{\rho_j} = 0$, we see that
\begin{align*}
    u = \sum_{j=0}^N a_j u_{\rho_j} = \sum_{j\neq i} (a_j - a_i) u_{\rho_j} \in \sigma_{i} \iff a_j - a_i\geq 0\ \forall\ j\neq i \iff a_i = \min_{j=0}^N \{a_j\}.
\end{align*}
This shows that, indeed, the curve class $\beta_\p$ constructed in \Cref{subsec: degree_basepoint} using any cone $\sigma_{i}$ containing $u$ equals $d_\p$, if $\vanishp = \emptyset$. A similar argument works if $\vanishp\neq\emptyset$.
\end{example}

\begin{figure}
    \centering
			\begin{tikzpicture}[scale = .8]
				\fill[red!30]  (0,0) -- (2,0) -- (2,2) -- (0,2);
				\fill[green!30]  (0,0) -- (0,2) -- (-2,2) -- (-2,-2);
				\fill[blue!30]  (0,0) -- (2,0) -- (2,-2) -- (-2,-2);
				\draw[gray] (-2,0)--(0,0);
				\draw[gray] (0,0)--(0,-2);
				\draw[thick] (0,0)--(0,2);
				\draw[thick] (0,0)--(2,0);
				\draw[thick] (0,0)--(-2,-2);
    		\draw (2.5,0) node[align=center] {$\rho_{1}$};
    		\draw (0,2.5) node[align=center] {$\rho_{2}$};
    		\draw (-2.5,-2.5) node[align=center] {$\rho_{0}$};
                \draw (1,1) node[align=center] {$\sigma_{\hat{0}}$};
    		\draw (-1,1) node[align=center] {$\sigma_{\hat{1}}$};
    		\draw (1,-1) node[align=center] {$\sigma_{\hat{2}}$};    
    		\end{tikzpicture}
    		\caption{The fan of $\bbP^{2}$.}
        \label{fig: fan P2}
\end{figure}

\begin{example}[Basepoints of quasimaps to products of projective spaces]\label{ex: basepoints_P}
Let 
\[
    \bbP = \bbP^{n_1}\times \ldots \times \bbP^{n_k}
\]
be a product of projective spaces. Quasimaps to $\bbP$ will be used in \Cref{sec: contraction} to define the contraction morphism between stable maps and stable quasimaps. 

Given a quasimap $\q = (C,L_\rho, s_\rho, c_m)$, we can use the isomorphisms $c_m$ to view it as a collection of $k$ line bundles $L_1, \ldots, L_k$ on $C$ together with a collection of $n_i+1$ sections $s_{0,i}, \ldots, s_{n_i,i} \in H^0(C,L_i)$ for each $i\in \{1,\ldots, k\}$. As in \Cref{ex: basepoints_PN}, we can view the line bundles $L_i$ and the sections $s_{j,i}$ as giving a rational morphism $\q\colon C \dashrightarrow \bbP$ in homogeneous coordinates. 

Let $\betap \in A_1(\bbP) \simeq \bbZ^k$ be the degree of $\q$ at $\p$. Note that a point $x\in C$ is a basepoint of $\q$ if and only if it is a basepoint of the composition $\pi_i \circ \q\colon C \dashrightarrow \bbP^{n_i}$ for some $i$, with $\pi_i\colon \bbP\to \bbP^{n_i}$ the projection in the $i$-th factor. Combining this observation and \Cref{ex: basepoints_PN}, we see that 
\[
    \betap = (d_{\p,1},\ldots, d_{\p,k}) \in \bbZ_{\geq 0}^k
\]
where
\[
    d_{\p,i} \coloneqq \min_{j=0}^{n_i} \{\ord_\p (s_{j,i})\} \in \bbZ_{\geq 0}
\]
is the degree of the quasimap $\pi_i\circ \q\colon C \dashrightarrow \bbP^{n_i}$ at $\p$.
\end{example}

\begin{example}[Basepoints of quasimaps to $\Bl_0\bbP^2$]\label{ex: basepoints_Bl}
We look at $X = \Bl_0\bbP^2$, whose fan is pictured in \Cref{fig: fan blowup2}. Let $\q=(C,L_\rho,s_\rho, c_m)$ be a quasimap to $\Bl_0\bbP^2$. To simplify the notation, we denote $L_{\rho_i}$ by $L_i$, $s_{\rho_i}$ by $s_i$ and the toric boundary divisor $D_{\rho_i}$ corresponding to $\rho_i$ by $D_i$. 

In $A_1(\Bl_0\bbP^2)$, we denote by $L$ the pullback of the class of a line in $\bbP^2$ and by $E$ the class of the exceptional divisor. Let $S=L-E$. Then $[D_0] = L$, $[D_1] =[D_2] =S$ and $[D_3] =E$. 

For certain choices of $m_0, m_1\in M$, we get isomorphisms $c_{m_0}\colon L_0^\vee \otimes L_1\otimes L_3\simeq \ccO_C$ and $c_{m_1}\colon L_1\otimes L_2^\vee \simeq \ccO_C$. We can use these isomorphisms to view $\q$ as a collection of two line bundles $L_1$ and $L_3$ on $C$, and four sections
\[
    s_0\in H^0(C,L_1\otimes L_3),\quad s_1,s_2\in H^0(C,L_1), \quad s_3\in H^0(C,L_3).
\]

\begin{figure}
	\centering
	\begin{tikzpicture}[scale = .8]
		\fill[red!30]  (0,0) -- (0,2) -- (-2,2);
		\fill[yellow!30] (0,0)--(2,0)--(2,2)--(0,2);
		\fill[green!30]  (0,0) -- (-2,2) -- (-2,-2) -- (0,-2);
		\fill[blue!30]  (0,0) -- (2,0) -- (2,-2) -- (0,-2);
		\draw[gray] (-2,0)--(0,0);
		\draw[thick] (0,0)--(0,2);
		\draw[thick] (0,0)--(2,0);
		\draw[thick] (0,0)--(-2,2);
		\draw[thick] (0,0)--(0,-2);
		\draw (0,2.5) node[align=center] {$\rho_{3}$};
		\draw (2.5,0) node[align=center] {$\rho_{1}$};
		\draw (-2.5,2.5) node[align=center] {$\rho_{2}$};
		\draw (0,-2.5) node[align=center] {$\rho_{0}$};
            \draw (1,1) node[align=center] {$\sigma_{1,3}$};
            \draw (-0.5,1.25) node[align=center] {$\sigma_{2,3}$};
            \draw (-1,-1) node[align=center] {$\sigma_{0,2}$};
            \draw (1,-1) node[align=center] {$\sigma_{0,1}$};
	\end{tikzpicture}
	\caption{The fan of $\Bl_0\bbP^2$.}
	\label{fig: fan blowup2}
\end{figure}

Suppose that $\q$ has degree $\beta = L$, therefore
\begin{equation}\label{eq: degrees Li}
    (\deg(L_i))_i = (L\cdot [D_i])_i = (1,1,1,0).
\end{equation}
Let $\p\in C$ be smooth and let $\ord_\p = (\ord_\p(s_i))_{i}$ be the collection of orders of vanishing at $\p$, with the convention that $\ord_\p(s_i) = \infty$ if $s_i$ is identically 0 on the component $\compp$ of $C$ containing $\p$. Note that 
\[
    \ord_\p \in \{0,1,\infty\} \times \{0,1,\infty\} \times \{0,1,\infty\} \times \{0,\infty\}
\]
by \Cref{eq: degrees Li}, so there are 54 different possibilities for $\ord_\p$. 

Some of those 54 are ruled out by generic non-degeneracy. We also impose that $\p$ is a basepoint, which happens if and only if $s_0(\p)=s_3(\p) = 0$ or $s_1(\p)=s_2(\p) = 0$. Finally, we look at all the possibilities modulo the symmetry interchanging $s_1$ and $s_2$. This leaves us with 14 different possibilities for the vector $\ord_\p$, which we collect in \Cref{tab: list betaps}. For each of them, we compute the vector $u$ in \Cref{eq: def u} (with the convention explained there if one of the entries is $\infty$), the collection of cones in $\sigma \in \Sigma(2)$ containing $u$, and the class $\betap = \beta(\p,\sigma)$.

\begin{table}
    \centering
    \begin{tabular}{ |c|c|c|c|c| } 
 \hline
  & $(\ord_\p(s_i))$ & $u=\sum_i \ord_\p(s_i) u_i$ & $\sigma$ & $\betap$\\\hline
    1 & $(0,1,1,0)$ & $u_3$ & $\sigma_{1,3}, \sigma_{2,3}$ & $E$ \\ \hline
    2 & $(0,1,1,\infty)$ & $u_3+\lambda_3 u_3$ & $\sigma_{1,3}, \sigma_{2,3}$ & $E$ \\ \hline
    3 & $(0,1,\infty,0)$ & $u_3 + \lambda_2 u_2$ & $\sigma_{2,3}$ & $E$ \\ \hline
    4 & $(0,1,\infty,\infty)$ & $\lambda_2u_2 + \lambda_3u_3$ & $\sigma_{2,3}$ & $E$ \\ \hline
    5 & $(1,0,0,\infty)$ & $u_1 + \lambda_3 u_3$ & $\sigma_{1,3}$ & $S$ \\ \hline   
    6 & $(1,0,1,\infty)$ & $u_2 + \lambda_3u_3$ & $\sigma_{2,3}$ & $S$ \\ \hline
    7 & $(1,0,\infty, 1)$ & $\lambda_2 u_2$ & $\sigma_{0,2}, \sigma_{2,3}$ & $S$ \\ \hline
    8 & $(1,1,1,0)$ & $0$ & $\sigma_{0,1}, \sigma_{0,2}, \sigma_{1,3}, \sigma_{2,3}$ & $L$ \\ \hline
    9 & $(1,1,1,\infty)$ & $\lambda_3u_3$ & $\sigma_{1,3}, \sigma_{2,3}$ & $L$ \\ \hline
    10 & $(1,1,\infty,0)$ & $\lambda_2u_2$ & $\sigma_{0,2}, \sigma_{2,3}$ & $L$ \\ \hline
    11 & $(1,1,\infty,\infty)$ & $\lambda_2 u_2 + \lambda_3 u_3$ & $\sigma_{2,3}$ & $L$ \\ \hline
    12 & $(\infty,0,\infty,1)$ & $\lambda_0 u_0 + \lambda_2 u_2$ & $\sigma_{0,2}$ & $S$ \\ \hline
    13 & $(\infty,1,1,0)$ & $\lambda_0 u_0$ & $\sigma_{0,1}, \sigma_{0,2}$ & $L$ \\ \hline
    14 & $(\infty,1,\infty,0)$ & $\lambda_0 u_0 + \lambda_2 u_2$ & $\sigma_{0,2}$ & $L$ \\ 
 \hline
\end{tabular}
    \vspace{1em}
    \caption{List of all possible $\betap \neq 0$ for a quasimap to $\Bl_0\bbP^2$ of degree $L$}
    \label{tab: list betaps}
\end{table}

To check that the table is correct, it is enough to see that for each entry one has
\[
    \ord_\p(s_i) \geq \betap \cdot [D_i]
\]
for all $i$, with equalities for some $i\in \{0,3\}$ and some $i \in \{1,2\}$. Note that $(S\cdot [D_{i}])_i = (1,0,0,1) $, $(E\cdot [D_{i}])_i = (0,1,1,-1) $ and $(L\cdot [D_{i}])_i = (1,1,1,0)$.

Finally, generalizing \Cref{ex: basepoints_PN,ex: basepoints_P}, we observe that the curve class $\betap$ at a smooth point $\p$ of any quasimap to $\Bl_0\bbP^2$ can also be described as follows: 
\[
    \betap = d_{\p,S} S + d_{\p,E} E
\]
with
\begin{align*}
    d_{\p,S} &= \min \{ \ord_\p (s_0), \ord_\p (s_1) + \ord_\p (s_3), \ord_\p (s_2) + \ord_\p (s_3)\}, \\
    d_{\p,E} &= \min \{\ord_\p (s_1), \ord_\p (s_2)\}.
\end{align*}
One way to see this is to use the closed embedding $\iota\colon \Bl_0 \bbP^2\hookrightarrow \bbP^2\times \bbP^1$ in \Cref{eq: embedding_blowup_P2xP1}, which satisfies that $\iota_\ast\colon A_1(\Bl_0\bbP^2)\to A_1(\bbP^2\times \bbP^1)$ maps $S$ to $(1,0)$ and $E$ to $(0,1)$, combined with \Cref{lem: basepoints of ibar q} and \Cref{ex: basepoints_P}. This argument uses that $\iota_\ast$ is injective and the compatibility of $\betap$ with pushforward along closed embeddings (\Cref{lem: pushforward degree of a basepoint}).
\end{example}

\subsection{Properties of the degree of a basepoint}\label{subsec: properties_degree_basepoint}

The following property is an immediate corollary of \Cref{prop: uniqueness of degree of a basepoint}.

\begin{corollary}\label{cor: expression regular extension}
    Let $\q = (C,L_\rho, s_\rho, c_m)$ be a quasimap to $X$ and let $B$ denote the set of basepoints of $\q$. For $\p\in B$, let $\coordp$ denote a local coordinate at $\p$. Then the regular extension $\qreg$ of $\q$ is given by
    \[
        (C,L'_\rho, s'_\rho, c'_m)
    \]
    with
    \begin{align*}
        L'_\rho &\coloneqq L_\rho\otimes \ccO_C(- \sum_{\p\in B}\betap\cdot [D_\rho]\ \p)\\
        s'_\rho &\coloneqq s_\rho\otimes \prod_{\p\in B} \coordp^{- \betap\cdot [D_\rho]},\\
        c'_m &= c_m\otimes \left(\bigotimes_{\p\in B} \psi_{-m,\p,\betap}\right).
    \end{align*}
\end{corollary}

\begin{corollary}\label{lem: equality of quasimaps}
    Let $\qone,\qtwo$ be quasimaps to $X$ with the same underlying curve. Then $\qone = \qtwo$ if and only if the following conditions hold
    \begin{enumerate}
        \item\label{item: same extension} $\qonereg=\qtworeg$
        \item\label{item: same basepoints} $B_{\qone} = B_{\qtwo}$ and
        \item\label{item: same degrees at basepoints} $\beta_{\p,\qone} = \beta_{\p,\qtwo}$ for all $\p\in B_{\qone} = B_{\qtwo}$.
    \end{enumerate}
\end{corollary}

\begin{proof}
    It follows from the explicit expression of $\qonereg$ and $\qtworeg$ in \Cref{cor: expression regular extension}. 
\end{proof}

Recall that a quasimap $\q$ has an associated regular extension $\qreg$ \Cref{def: regular extension}. In the following result we denote the degree $\beta_{\qreg}$ of $\qreg$ simply by $\beta_\reg$.

\begin{proposition}\label{prop: properties degree of a basepoint}
    Let $\q=(C,L_\rho,s_\rho,c_m)$ be a quasimap to $X$ of degree $\beta$. 
    The following holds
    \begin{enumerate}
        \item $\betap = 0$ if and only if $\p\notin B$,
        \item $\betap$ is effective for all $\p\in C$,
        \item\label{item: degree as regular + basepoints} $\beta = \beta_\reg + \sum_{\p\in B}\betap$.
    \end{enumerate}
\end{proposition}

\begin{proof}
\begin{enumerate}
    \item If $\p$ is singular, then $\betap = 0$ by \Cref{def: degree of a basepoint} and $\p\notin B$ by \Cref{def: quasimap}. For a smooth point $\p$, the result follows from \Cref{prop: uniqueness of degree of a basepoint}.
    \item The classes $\{[D_\rho]\colon \rho\in\Sigma(1)\}$ generate the effective cone 
    by \cite[Lemma 15.1.8]{CLS}. We know that $\betap = \beta(\p,\sigma)$ for some cone $\sigma$. By construction, the classes $\beta(\p,\sigma)$ lie in the dual of the cone of effective divisors, therefore also on the dual of the nef cone, so they must be effective by \cite[Theorem 6.3.22]{CLS}.
    \item Applying \Cref{prop: uniqueness of degree of a basepoint} successively to each basepoint, the result is a quasimap $\qpr$ without basepoints, which must then be the regular extension $\qreg$.\qedhere
\end{enumerate}
\end{proof}

Let $\iota\colon X \to Y$ be a closed embedding between smooth projective toric varieties. Recall the morphism $\ibar$ induced by $\iota$ between (prestable) toric quasimaps to $X$ and $Y$ (\Cref{constr: functoriality_qmaps} and \Cref{rmk: functoriality_same_maps_qmaps}).
We have the following compatibility with the degree of a basepoint.

\begin{lemma}\label{lem: degree of basepoint after composition} \label{lem: pushforward degree of a basepoint}
    Let $\ibar \colon \Qpre(X,\beta)\to \Qpre(Y,\iota_\ast \beta)$ , let $\q=(C,L_\rho,s_\rho,c_m)$ be a quasimap to $X$ of class $\beta$.
    For every 
    point $\p\in C$ we have the following equality in $A_1(Y)$:
    \[
        \iota_\ast(\betapq) = \beta_{\p,\ibar(\q)}.
    \]
\end{lemma}

\begin{proof}
    If $\p$ is not smooth, both terms are 0 by definition. If $\p$ is smooth, one can check, using the explicit descriptions in \Cref{prop: uniqueness of degree of a basepoint}, \Cref{rmk: coefficientes_a_explained} 
    and \Cref{eq: functoriality_qmaps}, that first twisting $\q$ by $\beta_{\p,\q}$ at $\p$ as in \Cref{prop: uniqueness of degree of a basepoint} and then applying $\ibar$ is the same as twisting $\ibar(\q)$ by $\iota_\ast \betapq$ at $\p$. Therefore, both $\iota_\ast \betapq$ and $\beta_{\p,\ibar(\q)}$ satisfy the uniqueness in \Cref{prop: uniqueness of degree of a basepoint} applied to $\ibar(\q)$. 
\end{proof}

\subsection{The relation between degree and length}

Given a prestable and not necessarily toric quasimap $\q$ with underlying curve $C$, the authors of \cite[Def. 7.1.1]{C-FKM} assign to every point $\p\in C$ a number $\ell(\p)$,
called the \textit{length of the quasimap $\q$ at the point $\p$}, as follows.

Given $X=W\GIT G$ with $W=\Spec(A)$, a character $\theta$ on $G$ and a quasimap $\q=(C,P,u)$ with $P$ a principal $G$-bundle on $C$ and $u$ a section of the fibre bundle $P\times_G W\to C$, let
    \begin{equation}\label{eq: CFKM_definition_length}
        \ell(\p)\coloneqq \min \left\{\frac{\ord_{\p}(u^*s)}{m} \colon s\in H^0(W,L_{m\theta})^G, u^*s\neq 0, m>0\right\}.
    \end{equation}
    If the character $\theta$ is chosen such that $H^0(W,L_\theta)^G$ generates $\oplus_{m\geq 0} H^0(W,L_{m\theta})^G$ as an algebra over $A^G$, then we can replace $\ord_{\p}(u^*s)/m$ by $\ord_{\p}(u^*s)$. Note that 
    \[
        \left\{s \colon s\in H^0(W,L_{m\theta})^G, m>0 \right\}
    \]
    is the ideal inside $A$ cutting out the scheme theoretic unstable locus $W^{\operatorname{us}}$ inside $W$.

Toric varieties have a natural GIT presentation, whose unstable locus can be described combinatorially. We use this fact to show that, for toric quasimaps, the degree $\betap$ recovers the length $\ell(\p)$.

\begin{lemma}\label{prop: degree_generalizes_length}
    Let $q$ be a quasimap to a smooth projective toric variety $X$ with underlying curve $C$. For every $x\in C$, the length $\ell(x)$ of $\q$ at $\p$ can be recovered combinatorially from the degree $\betap$ of $\q$ at $\p$ as follows:
    \begin{equation}\label{eq: length for toric}
        \ell(\p) = \min\left\{ \sum_{\rho\notin\sigma(1)} \betap \cdot [D_\rho] \colon \sigma \in \Sigma_{\max}, \vanishp\subseteq \sigma(1)\right\}.
    \end{equation}
\end{lemma}

\begin{proof}
    If $\p$ is singular, then $\betap=0$ by definition and $\ell(\p)=0$ because $\p$ is not a basepoint. Therefore, we can assume that $\p$ is smooth, so that $\betap$ is defined by the conditions in \Cref{prop: uniqueness of degree of a basepoint}.

    The toric variety $X_\Sigma$ has a quotient presentation, see \Cref{equation:GIT_presentation}. Let $\bbA^{\Sigma(1)} = \Spec(S)$, with $S = \bbC[z_\rho]_{\rho\in\Sigma(1)}$. For each cone $\sigma \in \Sigma$, let 
    \[
        z^{\hat{\sigma}} = \prod_{\rho\notin\sigma(1)} z_\rho \in S.
    \]
    If we choose a character $\theta$ on $G=\Hom(A^1(X),\bbC^*)$ coming from an ample divisor class, then the unstable locus $Z(\Sigma)$ inside $\bbA^{\Sigma(1)}$ is cut out by the ideal 
    \[
        B(\Sigma) = \left\langle z^{\hat{\sigma}} \colon \sigma\in\Sigma_{\max} \right\rangle \subset S.
    \]
    We can then rewrite \eqref{eq: CFKM_definition_length} for a prestable quasimap $\q = (C,L_\rho,s_\rho,c_m)$ to a toric variety $X_\Sigma$ as
    \begin{equation}\label{eq:expresion_length}
        \ell(\p) = \min \{ \ord_\p (s^{\hat{\sigma}}\mid_{\compp}) \colon \sigma\in\Sigma_{\max}, s^{\hat{\sigma}}\mid_{\compp} \neq 0\}
    \end{equation}
    where
    \begin{equation}\label{eq: def s to the sigma}
        s^{\hat{\sigma}} = \prod_{\rho\notin\sigma(1)} s_\rho \in H^0(C,\otimes_{\rho\notin \sigma(1)} L_\rho).
    \end{equation}
    Note that the section $s^{\hat{\sigma}}$ vanishes identically on $\compp$ if and only if for some $\rho\notin \sigma(1)$  the section $s_\rho$ vanishes identically on $\compp$. It follows that
    \begin{equation}\label{eq:vanishing_section_cone}
        s^{\hat{\sigma}}\mid_{\compp} \neq 0 \iff \vanishp\subseteq \sigma(1).
    \end{equation}

    Using \Cref{eq:expresion_length,eq: def s to the sigma,eq:vanishing_section_cone}, we have that
    \begin{equation}\label{eq: reformulation length}
        \ell(\p) = \min\left\{ \sum_{\rho\notin\sigma(1)} \ord_{\p}(s_\rho) \colon \sigma \in \Sigma_{\max}, \vanishp\subseteq \sigma(1)\right\}.
    \end{equation}
    To conclude, we need to show that \Cref{eq: reformulation length} equals the right hand side in \Cref{eq: length for toric}. One inequality follows immediately from \Cref{eq: inequality order betap}. On the other hand, in \Cref{eq: reformulation length} we can replace $\ord_{\p}(s_\rho)$ by $\beta(\p, \sigma) \cdot [D_\rho]$
    by the definition of $\beta(\p, \sigma)$ in  \Cref{constr: degree of basepoint no vanishing}. This finishes the proof since we saw in the proof of \Cref{prop: uniqueness of degree of a basepoint} that the degree $\betap$ equals $\beta(\p,\sigma)$ for some $\sigma \in \Sigma_{\max}$ with $\vanishp \subseteq \sigma(1)$.  
\end{proof}

\section{Embeddings of toric varieties and quasimap spaces}\label{sec: embeddings of toric varieties}

A morphism $\iota\colon X \to Y$ between smooth projective varieties induces a morphism
\[
    \imap: \Mgn(X,\beta)\to \Mgn(Y,\iota_\ast \beta),
\]
via composition, which is a closed embedding when $\iota$ is. Similarly, 
if $X$ and $Y$ are toric, $\iota$
induces a morphism
\[
    \ibar: \Qgn(X,\beta)\to \Qgn(Y,\iota_\ast \beta),
\]
described in \Cref{constr: functoriality_qmaps}. However, $\ibar$ may not be a closed embedding, even if $\iota$ is, see \Cref{ex: quasimaps do not embed}. Motivated by this example, we introduce a class of morphisms, that we call \textit{epic}, in \Cref{subsec: epic morphisms}. We show that $\ibar$ is a closed embedding if $\iota$ is an epic closed embedding in \Cref{cor: injective on A1 implies closed embedding}, and that every smooth projective toric variety admits an epic embedding into a product of projective spaces (\Cref{thm: toric is projectively epic}).

\subsection{Quasimaps do not embed along closed embeddings}\label{subsec: quasimaps do not embed}

Let $\iota\colon X\hookrightarrow Y$ be a closed embedding between smooth projective toric varieties. The induced morphism 
\[
    \ibar\colon \Qgn (X,\beta)\to \Qgn (Y,\iota_\ast \beta)
\]
is not a closed embedding in general. Here is an example, which appeared in \cite[Remark 2.3.3]{Battistella_Nabijou}, where it is attributed to Ciocan-Fontanine. 

\begin{example}\label{ex: quasimaps do not embed}
    Let $s\colon \bbP^1\times\bbP^1 \hookrightarrow \bbP^3$ be the Segre embedding, given in coordinates by 
	\[
		s([x\colon y], [z\colon w]) = [xz\colon xw\colon yz\colon yw].
	\]
    Note that $s$ is a toric closed embedding. Consider the following two quasimaps from $\bbP^1$, with homogeneous coordinates $[s\colon t]$, to $\bbP^1\times\bbP^1$ of degree $(2,2)$ 
	\begin{align*}
	    \qone([s\colon t]) &= [s^2\colon st],[st\colon t^2],& 
		\qtwo([s\colon t]) &= [st\colon t^2],[s^2\colon st].
	\end{align*}
    Both of them have the same image under $\Qs$
	\[
		\Qs(\qone) ([s\colon t]) = \Qs(\qtwo) ([s\colon t]) = [s^3t\colon s^2t^2\colon s^2t^2 \colon st^3],
	\]
    which is a quasimap to $\bbP^3$ of degree $4$. This means that $\Qs$
    is not a monomorphism, therefore it is not a closed embedding. Note that the quasimaps above are not stable, but the same conclusion holds if we add marked points. 
    
    Our contribution is the following: we can use the degree of a basepoint (see \Cref{def: degree of a basepoint}) to explain why $\Qs$ is not a monomorphism.  Both $\qone$ and $\qtwo$ have the same regular extension $\qreg$, namely, the diagonal embedding of $\bbP^1$ in $\bbP^1 \times \bbP^1$,
    \[
        \qreg([s\colon t]) = [s\colon t], [s\colon t].
    \]
    Let $0=[1\colon 0]$ and let $\infty = [0\colon 1]$ in $\bbP^1$. Then $\qone$ and $\qtwo$ both have the same basepoints $B_\qone = B_\qtwo = \{0,\infty\}$, but we can still distinguish $\qone$ from $\qtwo$ because the degrees at the basepoints differ:
    \begin{align*}
        \beta_{0, \qone} &= (0,1), & \beta_{\infty, \qone} &= (1,0),\\
        \beta_{0, \qtwo} &= (1,0), & \beta_{\infty, \qtwo} &= (0,1)
    \end{align*}
    if we identify $A_1(\bbP^1\times \bbP^1)$ with $\bbZ^2$. Applying $\Qs$, both quasimaps not only have the same regular extension and the same set of basepoints, but also
    \begin{equation}\label{eq: equality_degrees_example}
        \beta_{0, \Qs(\qone)} = \beta_{\infty, \Qs(\qone)} = \beta_{0, \Qs(\qtwo)} = \beta_{\infty, \Qs(\qtwo)} = 1 \in A_1(\bbP^3).
    \end{equation}
    Therefore, we can no longer distinguish them, by \Cref{lem: equality of quasimaps}. Note that, by \Cref{lem: degree of basepoint after composition}, the degree of a basepoint is compatible with $\Qs$ and $s_\ast$, in the sense that $\beta_{x, \Qs(\q_i)} = s_\ast \beta_{x,\q_i}$; therefore, the equalities in \Cref{eq: equality_degrees_example} follow from the fact that $s_\ast(1,0) = s_\ast(0,1) = 1$.
\end{example}

\Cref{ex: quasimaps do not embed} suggests that the reason why $\Qs$ is not a monomorphism (and therefore the reason why it is not a closed embedding, since it is proper) is that $s_\ast \colon A_1(\bbP^1\times\bbP^1)\to A_1(\bbP^3)$ is not injective. This observation motivates \Cref{thm: fibres of ibar} and \Cref{cor: injective on A1 implies closed embedding}.

\subsection{Epic morphisms}\label{subsec: epic morphisms}

Motivated by \Cref{ex: quasimaps do not embed}, we introduce the following class of morphisms.

\begin{proposition}\label{prop: equivalence epic morphisms}
    Let $X$ and $Y$ be smooth projective varieties whose Picard group is free of finite rank and let   $f\colon X \to Y$. The following are equivalent:
    \begin{enumerate}
        \item\label{item: epi pic} $f^*\colon \Pic(Y) \to \Pic(X)$ is surjective,
        \item\label{item: epi A1} $f^*\colon A^1(Y) \to A^1(X)$ is surjective,
        \item\label{item: mono A1} $f_*\colon A_1(X) \to A_1(Y)$ is injective.
    \end{enumerate}
\end{proposition}

\begin{proof}
    The equivalence between \Cref{item: epi pic} and \Cref{item: epi A1} holds because for a smooth variety $\Pic$ and $A^1$ are isomorphic. 
    Finally, \Cref{item: epi A1} and \Cref{item: mono A1} follows from the assumption that the Picard groups are free of finite rank, because $f_\ast$ and $f^\ast$ are transpose by the projection formula.
\end{proof}

\begin{definition}\label{def: epic morphism}
    A morphism $f \colon X \to Y$ of smooth projective varieties whose Picard group is free of finite rank is \textit{epic} if it satisfies any of the equivalent conditions in \Cref{prop: equivalence epic morphisms}. The name is motivated by \Cref{item: epi pic} in \Cref{prop: equivalence epic morphisms}: epic morphisms induce an epimorphism on Picard groups.
\end{definition}

We are particularly interested in varieties that admit an epic closed embedding in a product of projective spaces.

\begin{proposition}\label{prop: equivalence projectively epic}
    Let $X$ be a smooth projective variety such that $\Pic(X)$ is free of finite rank. The following are equivalent:
    \begin{enumerate}
        \item\label{item: exists epic embedding} there exist $k, n_1, \ldots, n_k \geq 1$ and an epic closed embedding $X \hookrightarrow \bbP^{n_1}\times \ldots \times \bbP^{n_k}$,
        \item\label{item: exists bpf generating set} $X$ admits a generating set of $\Pic(X)$ consisting of basepoint free divisors.
    \end{enumerate}
\end{proposition}

\begin{proof}
    Suppose \Cref{item: exists epic embedding} holds. Let $\bbP = \bbP^{n_1}\times \ldots \times \bbP^{n_k}$ and let $L_i\in \Pic(X)$ be the pullback of $\ccO_{\bbP}(\bbP^{n_1}\times \ldots \times H_i \times \ldots \times \bbP^{n_k})$ with $H_i$ a hyperplane in $\bbP^{n_i}$. Then each $L_i$ is basepoint free and $\{L_i\}$ generate $\Pic(X)$ because the embedding is epic, so \Cref{item: exists bpf generating set} holds.

    Conversely, let $\ccS$ be a generating set $\ccS$ of $\Pic(X)$ whose elements are basepoint free, let $[A]\in \Pic(X)$ be a very ample class and let $\ccS' = \ccS\cup\{[A]\}$. Consider the induced morphism
		\[
			\iota \colon X \to \bbP\coloneqq \prod_{[D]\in \ccS'} \bbP(H^0(\ccO_X(D)))
		\]
	associated to the complete linear systems $\lvert D \rvert$, for $ [D]\in \ccS'$. Then $\iota^\ast$ is surjective on Picard groups because $\ccS'$ is a generating set. Finally, the diagonal arrow in the following diagram is a closed embedding, which implies that so is $\iota$.
	\[
		\begin{tikzcd}
			X\arrow{r}{\iota}\arrow{dr}[']{\lvert A \rvert} & \bbP \arrow{d}{\pi_A}\\
			& \bbP(H^0(\ccO_X(A)))
		\end{tikzcd}\qedhere
	\] 
\end{proof}

\begin{definition}\label{def: projectively epic}
    A smooth projective variety whose Picard group is free of finite rank is called \textit{projectively epic} if it satisfies any of the equivalenc conditions in \Cref{prop: equivalence projectively epic}.
\end{definition}

\begin{theorem}\label{thm: toric is projectively epic}\label{prop: embedding with surjection on Pic}\label{cor: embedding in P with iso on A1}
    Every smooth projective toric variety is projectively epic.
\end{theorem}

\begin{proof}
    The nef and Mori cones of a smooth projective toric variety $X$ are dual strongly convex rational polyhedral cones of full dimension by \cite[Theorem 6.3.12]{CLS}. In particular, the semigroup of divisor classes which are nef admits a finite generating set $\ccS$ by Gordan's lemma, which also generates $\Pic(X)$. Furthermore, the elements of $\ccS$ are basepoint free by \cite[Theorem 6.3.12]{CLS}, therefore \Cref{prop: equivalence projectively epic} holds, so $X$ is projectively epic.
\end{proof}

\begin{example}\label{ex: embeddings}
    Let $X = \Bl_0 \bbP^2$ be the blow-up of $\bbP^2$ at the origin, which is a toric variety whose fan is pictured in \Cref{fig: fan blowup2}. Let $\pi\colon X \to \bbP^2$ be the natural projection and let $H$ be the hyperplane class in $\Pic(\bbP^2)$. The group $\Pic(X)$ is generated by $L\coloneqq \pi^*H$ and the class $E$ of the exceptional divisor. Let $S = L-E$, which is the class of the strict transform of a line in $\bbP^2$. Then the nef cone of $X$ is generated by $S$ and $L$. 
    
    There are four toric boundary divisors in $X$, which we denote by $D_i = D_{\rho_i}$, with the conventions in \Cref{fig: fan blowup2}. They satisfy that
    \begin{align*}
        [D_0] &=L, &  
        [D_1] = [D_2] &=S, &
        [D_3] &=E. 
    \end{align*}
    Using this, together with the description of global sections of toric line bundles, one can easily describe in coordinates the following closed embeddings of $X$ into products of projective spaces induced by complete linear systems. Such morphisms will be written using the ring $\bbC[x_0,x_1,x_2,x_3]$ of homogeneous coordinates on $X$, with $x_i$ corresponding to $\rho_i$.

    The set $\{S,L\}$ generates the nef cone of $X$. Although none of the two classes is ample, the morphism
    \[
        \iota \colon X\hookrightarrow \bbP(H^0(\ccO_X(L)) \times \bbP(H^0(\ccO_X(S)) = \bbP^2\times \bbP^1
    \]
    is an epic closed embedding. One way to see that $\iota$ is indeed a closed embedding is to realize that this is the usual description of $X$ as a closed subvariety of $\bbP^2\times\bbP^1$. Its expression in coordinates is
    \begin{equation}\label{eq: embedding_blowup_P2xP1}
        \iota(x_0,x_1,x_2,x_3) = [x_0\colon x_1x_3 \colon x_2x_3], [x_1\colon x_2].
    \end{equation}
    This example illustrates that, in the notation of the proof of \Cref{prop: equivalence projectively epic}, it may not be necessary that $\ccS'$ contains a very ample class if the elements of $\ccS$ already induce a closed embedding.

    The set $\{S,S + L\}$ also generates $\Pic(X)$, both classes are nef and $S+L$ is ample (thus very ample).
    It induces a closed embedding
    \[
        i \colon X\hookrightarrow \bbP(H^0(\ccO_X(S)) \times \bbP(H^0(\ccO_X(S+L)) = \bbP^1\times \bbP^4
    \]
    with
    \[
        i(x_0,x_1,x_2,x_3) = [x_1\colon x_2], [x_0x_2\colon x_0x_1\colon x_2^2x_3\colon x_1x_2x_3 \colon x_1^2x_3].
    \]

    Similarly, one could take the set $\{S+L, L\}$, which also generates $\Pic(X)$, with $L$ nef and $S+L$ ample. The corresponding closed embedding is
    \[
        j \colon X\hookrightarrow \bbP(H^0(\ccO_X(S+L)) \times \bbP(H^0(\ccO_X(L)) = \bbP^4\times \bbP^2
    \]
    with
    \[
        j(x_0,x_1,x_2,x_3) = [x_0x_2\colon x_0x_1\colon x_2^2x_3\colon x_1x_2x_3 \colon x_1^2x_3], [x_0,x_2x_3,x_1x_3].
    \]
\end{example}

\subsection{A criterion for quasimaps to embed along an embedding}\label{subsec: criterion_qmaps_embedd}

We fix a closed embedding $\iota\colon X\hookrightarrow Y$ between smooth projective toric varieties for the rest of \Cref{subsec: criterion_qmaps_embedd} and we study the morphism
\[
    \ibar\colon \Qgn (X,\beta)\to \Qgn (Y,\iota_\ast \beta)
\]
on closed points. For that, we start with two lemmas, which apply more generally to (prestable) toric quasimaps, see \Cref{rmk: functoriality_same_maps_qmaps}.\\

Recall that we denote by $B_\q$ the set of basepoints of $\q$ (see \Cref{def: quasimap}). In \Cref{def: regular extension}, we introduced the notion of regular extension $\qreg$ of a quasimap $\q$ to $X$. In this section, we denote $\qreg$ by $r_X(\q)$. Similarly, $r_Y(\q')$ denotes the regular extension of a quasimap $\q'$ to $Y$.

\begin{lemma}\label{lem: commutativity diagram embedding and extension}
    For every quasimap $\q$ to $X$, we have that
    \[
        r_Y(\ibar \q) = \imap(r_X \q).
    \]
\end{lemma}

\begin{proof}
    The two maps have the same underlying curve $C$ by \Cref{rmk: no_stabilization} and they are both extensions of the map $\iota\circ (\q\mid_{C\setminus B})$ from the dense open $C\setminus B$ to $C$, so they must agree.
\end{proof} 

\begin{lemma}\label{lem: basepoints of ibar q}
    For every quasimap $\q$ to $X$, we have that
    \[
        B_\q = B_{\ibar (\q)}.
    \]
\end{lemma}

\begin{proof}
    Follows from \Cref{lem: pushforward degree of a basepoint} and \Cref{prop: properties degree of a basepoint} because the pushforward along a closed embedding of an effective non-zero class is a non-zero class.
\end{proof}

\begin{theorem}\label{thm: fibres of ibar}
    Let $\beta\in A_1(X)$ be an effective non-zero curve class and let $\q$ be an $n$-marked genus-$g$ quasimap to $Y$ of class $\iota_\ast \beta$ 
    with underlying curve $C$. Consider the morphism induced by $\iota$ on prestable quasimaps
    \[
        \ibar \colon \Qpre(X,\beta)\to \Qpre(Y,\iota_\ast \beta).
    \]
    Then the fibre $\ibar^{-1}(\q)$ of $\q$ along $\ibar$ is non-empty if and only if 
    \begin{enumerate}
        \item\label{item: regular extension factors} there exists a morphism $f\colon C \to X$ such that $\qreg = \iota \circ f$ and 
        \item\label{item: existence curve class + conditions} there exist effective non-zero classes $(\beta^\p)_{\p\in B_\q}$ in $A_1(X)$ such that
        \begin{enumerate}
            \item\label{item: pushforward agrees} $\iota_\ast(\beta^\p) = \betapq$ for all $\p\in B_\q$,
            \item\label{item: correct sum} $\beta = \beta_f + \sum_{\p\in B_\q} \beta^\p$ and
            \item\label{item: order condition} $\ord_\p(t_{\rho}) + \beta^\p \cdot D_\rho \geq 0$ for every $\rho\in \Sigma_X(1)$ and every $\p\in B_\q$
        \end{enumerate}
    \end{enumerate}
    with $\{t_\rho\}_{\rho\in\Sigma_X(1)}$ the underlying sections of $f$.

    Furthermore, the fibre $\ibar^{-1}(\q)$ is in one-to-one correspondence with the (finite) set of effective non-zero classes $(\beta^\p)_{\p\in B_\q}$ in $A_1(X)$ satisfying \Cref{item: pushforward agrees,item: correct sum,item: order condition}.

    If, moreover, $\q$ is stable, then so is every element in $\ibar^{-1}(\q)$.
\end{theorem}

\begin{proof}
    Firstly, let $\qone$ be a quasimap to $X$ of class $\beta$ such that $\ibar\qone = \q$. By \Cref{lem: commutativity diagram embedding and extension},
    \[
        \imap(r_X \qone) = r_Y(\ibar \qone) = r_Y(\q).
    \]
    This means that $r_Y(\q)=\qreg$ factors through $X$ and we can take $f \coloneqq r_X(\qone) = \qonereg$. Using \Cref{lem: basepoints of ibar q}, the classes $\beta^\p \coloneqq \beta_{\p,\qone}$, for $\p\in B_q$, satisfy the desired conditions. Indeed, \Cref{item: pushforward agrees} follows from \Cref{lem: degree of basepoint after composition}, \Cref{item: correct sum} from \Cref{prop: properties degree of a basepoint} and \Cref{item: order condition} from \Cref{cor: expression regular extension}, since 
    \[
        \ord_\p(t_{\rho}) + \beta_{\p,\qone} \cdot D_\rho = \ord_{\p}(s_\rho^{(1)}) \geq 0,
    \]
    where $s_\rho^{(1)}$ denote the sections of $\qone$.

    Conversely, let $\q$ be a quasimap to $Y$ of class $\iota_\ast \beta$ and let $f$ and $\beta^\p$, for $\p\in B_\q$, be as in the statement.  
    Since $\imap$ is a closed embedding, the morphism $f$ is determined uniquely by the condition $\qreg = \iota\circ f$. We show how to use the data in \Cref{item: existence curve class + conditions} to construct a unique quasimap $\qone$ to $X$ of class $\beta$ such that $\ibar\qone = \q$. We fix that $\qonereg = f$, and then we twist the line bundle-section pairs associated to $f$ at each basepoint $\p\in B_\q$ as in \Cref{cor: expression regular extension}, but changing the minus sign by a plus sign. This makes $\p$ a basepoint of $\qonereg$ with class $\beta^\p$ for every $\p\in B_\q$. By \Cref{lem: equality of quasimaps}, the previous construction clearly determines a quasimap $\qone$ uniquely.
    
    Finally, we show that if $\q$ is stable, then so is $\qone$. This follows from the fact that $\q$ and $\qone$ have the same curve (and marks) and the fact that components contracted by $\qone$ are also contracted by $\q$. Indeed, if a component $C'$ of the underlying curve of $\q$ and $\qone$ has degree $\beta_{C',\qone} = 0$ with respect to $\qone$, then by \Cref{prop: properties degree of a basepoint} and \Cref{lem: degree of basepoint after composition} we have that $\beta_{C',\q} = i_\ast \beta_{C',\qone} = 0$. But then stability of $\q$ ensures that $C'$ must have enough special points.
\end{proof}

\begin{corollary}\label{cor: injective on A1 implies closed embedding}
    Let $\iota\colon X\hookrightarrow Y$ be an epic closed embedding. Then
    the morphism
    \[
        \ibar \colon \Qgn(X,\beta)\to \Qgn(Y,\iota_\ast \beta)
    \]
    induced by $\iota$ is a closed embedding over $\bbC$. 
\end{corollary}

\begin{proof}
    The morphism $\ibar$ is proper over $\Spec(\bbC)$ by \cite[Lemma 01W6]{stacks-project} because $\ccQ (X,\beta)$ is proper and $\ccQ (Y,\iota_\ast \beta)$ is separated, both over $\Spec(\bbC)$. By \cite[Lemma 04XV]{stacks-project}, it is enough to show that $\ibar$ is also a monomorphism.
    
    Since $\Qgn(X,\beta)$ is locally of finite type over $\Spec(\bbC)$, so is the morphism $\ibar$ by \cite[Lemma 01T8]{stacks-project}. Furthermore, the assumption that $\iota_\ast$ is injective on curve classes ensures, by \Cref{thm: fibres of ibar}, that for every (geometric) point $s$ in $\Qgn(Y,\iota_\ast \beta)$, the natural morphism $\Qgn(X,\beta)_s \to s$ induced by $\ibar$ is an isomorphism. This ensures that $\ibar$ is a monomorphism by \cite[Lemma 05VH]{stacks-project}.
\end{proof}

\section{The contraction morphism for smooth projective toric varieties}\label{sec: contraction}

We fix a smooth projective toric variety $X$ for the rest of \Cref{sec: contraction}.

\subsection{Construction of the contraction morphism}\label{subsec: construction of contraction}

We construct a morphism of stacks
\[
    c = c_{X}\colon \cssgn \rightarrow \Qgn(X,\beta),
\]
with $\cssgn$ a closed substack of $\Mgn(X,\beta)$. The morphism $c_X$, called the \textit{contraction morphism} of $X$, will be a generalization of the contraction morphism $c_{\bbP^N}$ described in \Cref{subsec: contraction for Pn} and it extends the identity on the locus of stable maps which are stable as quasimaps. The construction of $c_X$ relies on results from \Cref{sec: embeddings of toric varieties} and $c_{\bbP^N}$.\\

\begin{remark}
    In \Cref{subsec: contraction for Pn} we have reviewed the morphism $c_{\bbP^N}$. One can check that a similar construction defines a (global) contraction morphism for products of projective spaces.
\end{remark}

\begin{construction}\label{constr: c_X}
    Let $\beta\in A_1(X)$ be an effective curve class. Choose an epic closed embedding
    \[
		\iota \colon X \hookrightarrow \bbP\coloneqq \bbP^{n_1} \times \ldots \times \bbP^{n_k}.
    \]
    Such an embedding always exists by \Cref{cor: embedding in P with iso on A1}. From \Cref{cor: injective on A1 implies closed embedding} it follows that
    \[
        \ibar \colon \Qgn(X,\beta)\hookrightarrow \Qgn(\bbP,\iota_\ast \beta)
    \]
    is a closed embedding. Let 
    \[
        c_\bbP\colon \Mgn(\bbP,\iota_\ast\beta) \to \Qgn(\bbP,\iota_\ast\beta)
    \]
    denote the contraction morphism for $\bbP$.
    Then $c_\bbP$ fits in the following diagram.
    \begin{equation}\label{eq: diagram that defines c_X}
    \begin{tikzcd}
        \Mgn(X,\beta) \arrow[hook]{r}{\imap}&
        \Mgn(\bbP,\iota_\ast\beta) \arrow{d}{c_\bbP}\\
        \Qgn(X,\beta) \arrow[hook]{r}{\ibar} &
        \Qgn(\bbP,\iota_\ast\beta)
    \end{tikzcd}
    \end{equation}
    We define the stack $\cssgn$ by the following Cartesian diagram
    \[
        \begin{tikzcd}
            \cssgn \arrow{r}{j} \arrow{d}{c_X}    & \Mgn(X,\beta) \arrow{d}{c_\bbP \circ \imap} \\
            \Qgn(X,\beta)\arrow{r}{\ibar}   & \Qgn(\bbP,\iota_\ast \beta)    
        \end{tikzcd}
    \]
    and we call the induced morphism 
    \[
        c_X\colon \cssgn \to \Qgn(X,\beta)
    \]
    the \textit{contraction morphism of $X$}. 
\end{construction}

Note that $j\colon \cssgn\to \Mgn(X,\beta)$ is a closed embedding and $c_X\colon \cssgn \to \Qgn(X,\beta)$ is locally of finite type.

The definition of $\cssgn$ and of $c_X$ a priori depends on the chosen epic closed embedding $\iota$.  Surprisingly, the description of the closed points of $\cssgn$ in \Cref{prop: description css and c_X} is independent of $\iota$. Equivalently, for a morphism $f$ to $X$, the property that $(c_\bbP \circ \imap) (f)$ factors through $\ibar$ is independent of the chosen epic embedding $\iota$ in a product of projective spaces. In future work, we plan to study the image of $\ibar$ in order to understand this property better. We expect that a similar result for families would have applications in enumerative geometry.

\subsection{Description of contraction on points}\label{subsec: description_contraction}

We describe the points of the closed substack $\cssgn$ inside $\Mgn(X,\beta)$ and the morphism $c_X$ from \Cref{constr: c_X} in terms of line bundle-section pairs. 
We use the notations introduced in \Cref{subsec: functoriality quasimaps,subsec: construction of contraction}.\\

\begin{proposition}\label{prop: description css and c_X}
    Let $(C,L_\rho, s_\rho, c_m)$ be the $\Sigma$-collection corresponding to an $n$-marked genus-$g$ stable map $f\colon C\to X$. Let $T_1,\ldots, T_\ell$ be the rational tails of $f$, let $\tilde{C}$ be the curve obtained by contracting all the rational tails in $C$, viewed as a subcurve of $C$. For $1\leq i \leq \ell$, let $\p_i = T_i\cap \overline{(C\setminus T_i)}$ and let $\coord_i$ denote a local coordinate at $\p_i$ inside $\tilde{C}$.
    Then
    \begin{enumerate}[(i)]
        \item \label{item: condition defining css} The map $(C,L_\rho, s_\rho, c_{m_X})$ belongs to $\cssgn$ if and only if
        \[
            \ord_{\p_i}(s_\rho\mid_{\tilde{C}}) + \beta_{T_i} \cdot D_\rho \geq 0
        \]
        for every $\rho\in\Sigma_X(1)$ and every $1\leq i\leq \ell$.
        
        \item\label{item: description of c_X} If $(C,L_\rho, s_\rho, c_{m_X})$ lies in $\cssgn$ then its image under $c_X$ is the quasimap $(\tilde{C},\tilde{L}_\rho, \tilde{s}_\rho, \tilde{c}_{m_X})$ with
        \begin{align*}
            \tilde{L}_\rho &= L_\rho\mid_{\tilde{C}} \otimes \ccO_{\tilde{C}}\left(\sum_{i=1}^\ell (\beta_{T_i}\cdot D_\rho) \p_i\right), \\ 
            \tilde{s}_\rho &= s_\rho\mid_{\tilde{C}} \otimes \prod_{i=1}^\ell \coord_i^{\beta_{T_i}\cdot D_\rho},\\
            \tilde{c}_{m_X} &= c_{m_X}\mid_{\tilde{C}}\otimes (\otimes_{i=1}^\ell \psi_{m_Y, x_i, \beta_{T_i}}),
    \end{align*}
    where 
    the isomorphisms $\psi_{m_X,x_i,\beta_{T_i}}$ are constructed 
    in \Cref{rmk: compatible isos}.
\end{enumerate}
\end{proposition}

\begin{proof} 
    The proof of \Cref{item: description of c_X} follows from a careful study of the morphisms 
    in Diagram \eqref{eq: diagram that defines c_X}.
    The morphism $c_{\bbP}$ is a generalization of the morphism $c_{\bbP^N}$, which is described in terms of line bundle-section pairs in \Cref{subsec: contraction for Pn}. Similarly, $\imap$ is described in \Cref{constr: functoriality_qmaps} (see \Cref{rmk: functoriality_same_maps_qmaps}). These descriptions imply \Cref{item: description of c_X}.

    For completeness, we spell out the details. Let us write 
    \begin{align*}
        (c_\bbP \circ \imap) (C,L_\rho,s_\rho,c_{m_X}) &= (\hat{C}, \hat{L}_\tau, \hat{s}_\tau, \hat{c}_{m_Y}),\\
        \ibar (\tilde{C},\tilde{L}_\rho, \tilde{s}_\rho, \tilde{c}_{m_X}) &= (\tilde{C},\overline{L}_\tau, \overline{s}_\tau, \overline{c}_{m_Y}).
    \end{align*}
    It suffices to show the equality
    \[
        (\hat{C}, \hat{L}_\tau, \hat{s}_\tau, \hat{c}_{m_Y}) = (\tilde{C},\overline{L}_\tau, \overline{s}_\tau, \overline{c}_{m_Y}).
    \]
    Furthermore, in that case it follows that $(\tilde{C},\tilde{L}_\rho, \tilde{s}_\rho, \tilde{c}_{m_X})$ is generically non-degenerate, because $(\tilde{C},\overline{L}_\tau, \overline{s}_\tau, \overline{c}_{m_Y})$ is, by \Cref{lem: basepoints of ibar q}.

    The equality $\hat{C} = \tilde{C}$ follows from the description of $c_\bbP$ and the fact that $\iota$ is a closed embedding, so $\imap$ and $\ibar$ preserve the underlying curve.  

    For the line bundles, we have that
    \begin{align*}
        \hat{L}_\tau &= \bigotimes_\rho L_\rho^{a_\rho^\tau}\mid_{\tilde{C}} \otimes \left( \bigotimes_{i=1}^\ell \ccO_{\tilde{C}}(\iota_\ast (\beta_{T_i}) \cdot D_\tau\ \p_i)\right),\\
        \overline{L}_\tau&= \bigotimes_\rho L_\rho^{a_\rho^\tau}\mid_{\tilde{C}} \otimes \left( \bigotimes_{i=1}^\ell \ccO_{\tilde{C}}(\beta_{T_i} \cdot (\sum_{\rho} a_\rho^\tau D_\rho) \ \p_i)\right).
    \end{align*}
    with $a_{\rho}^\tau$ as in \Cref{constr: functoriality_qmaps}. 
    The desired equality follows from the projection formula and \Cref{rmk: coefficientes_a_explained}, since
    \[
        \iota_\ast (\beta_{T_i}) \cdot D_\tau = \beta_{T_i} \cdot \iota^\ast (D_\tau) = \beta_{T_i} \cdot (\sum_{\tau} a_\rho^\tau D_\rho).
    \]

    For the sections, we have that
    \begin{align*}
        \hat{s}_\tau&= \mu_{\underline{a}^\tau} \prod_\rho s_\rho^{a_\rho^\tau}\mid_{\tilde{C}} \prod_{i=1}^\ell \coord_i ^{\iota_\ast(\beta_{T_i})\cdot D_\tau}\\
        \overline{s}_\tau&= \mu_{\underline{a}^\tau} \prod_\rho s_\rho^{a_\rho^\tau}\mid_{\tilde{C}} \prod_{i=1}^\ell \coord_i ^{\beta_{T_i}\cdot (\sum_{\rho} a_\rho^\tau D_\tau)},
    \end{align*}
    and we conclude by the same argument.

    Finally, for the isomorphisms we have that
    \begin{align*}
        \hat{c}_{m_Y}&= c_{m_X} \otimes \left(\bigotimes_i \psi_{m_Y,x_i,\iota_\ast(\beta_{T_i})}\right)\\
        \overline{c}_{m_Y}&= c_{m_X} \otimes \left(\bigotimes_i \psi_{m_X,x_i,\beta_{T_i}}\right)
    \end{align*}
    where $m_X \in M_X$ is uniquely determined by the conditions
    \[
        \scalar{m_X,u_\rho} = \sum_\tau a_\rho^\tau \scalar{m_Y,u_\tau}
    \]
    for every $\rho\in\Sigma(1)$.
    One can check that this conditions implies the equality $\psi_{m_Y,x_i,\iota_\ast(\beta_{T_i})} = \psi_{m_X,x_i,\beta_{T_i}}$ for each $i$.
  
    To conclude the proof, we show \Cref{item: condition defining css}.
    By construction of $\cssgn$, we need to find under what conditions the quasimap $q = (c_\bbP\circ \imap)(f)$ factors through the image of $\ibar$, for which we use \Cref{thm: fibres of ibar}. It suffices to rewrite \Cref{item: regular extension factors,item: existence curve class + conditions} in \Cref{thm: fibres of ibar} for $q$ in terms of $f$ itself.

    \Cref{item: regular extension factors} in \Cref{thm: fibres of ibar} is automatic here. Indeed, $q_\reg = (c_\bbP(\iota\circ f))_{\reg}$ and $\iota \circ (f\mid_{\tilde{C}})$ are maps which agree on a dense open in $\tilde{C}$, therefore they must agree everywhere. This implies that $q_\reg$ factors through $X$.

    Let us conclude with \Cref{item: existence curve class + conditions} in \Cref{thm: fibres of ibar}. By the proof of \Cref{thm: fibres of ibar}, if $\q=\ibar \qone$ for some quasimap $\qone$ to $X$, then $\beta^{\p_i} = \beta_{\p_i,\qone}$. Combining this with \Cref{item: pushforward agrees} and the equality $\beta_{\p_i,\q} = \iota_\ast(\beta_{T_i,f})$, it follows $\iota_\ast(\beta^{\p_i}) = \iota_\ast(\beta_{T_i,f})$, and therefore $\beta^{\p_i} = \beta_{T_i,f}$ since $\iota_\ast$ is injective. This is the only possible choice for the class $\beta^{\p_i}$.
    Therefore, $f$ lies in $\cssgn$ if and only if \Cref{item: correct sum,item: order condition} of \Cref{thm: fibres of ibar} hold for this choice of $\beta^{\p_i}$. \Cref{item: correct sum} says
    \[
        \beta_f = \beta_{f\mid_{\tilde{C}}} + \sum_{i=1}^\ell \beta_{T_i,f},
    \]
    which is immediate. Thus, the only non-trivial condition for $f$ to lie in $\cssgn$ is \Cref{item: order condition}, which says that
    \[
        \ord_{\p_i}(s_\rho\mid_{\tilde{C}}) + \beta_{T_i,f}\cdot D_\rho \geq 0
    \]
    for every $\rho\in\Sigma(1)$ and every $1\leq i \leq \ell$. This is exactly the condition in \Cref{item: condition defining css}.
\end{proof}

Note that, if $X$ has the property that all the toric boundary divisors $D_\rho$ are nef, then $c_X$ is defined on the whole of $\Mgn(X,\beta)$ since the condition in \Cref{item: condition defining css} in \Cref{prop: description css and c_X} holds trivially.

\begin{example}\label{ex:family_special_fibre_V}
    In order to illustrate the content of \Cref{prop: description css and c_X}, we construct a family of stable maps whose special fibre lies in $\cssgn$ but whose general fibre lies outside $\cssgn$.
    
    Let $X=\Bl_0 \bbP^2$ be the blow-up of $\bbP^2$ at the origin, whose fan is pictured in \Cref{fig: fan blowup2}. We will use the curve classes $L,S, E\in A_1(\Bl_0 \bbP^2)$ introduced in \Cref{ex: embeddings}. By \cite{Cox_functor}, a morphism to $\Bl_0 \bbP^2$ is given by line bundle-section pairs $(L_i,s_i)$ for $i=0,\ldots, 3$ satisfying the following conditions
    \begin{enumerate}
        \item $L_0\simeq L_1\otimes L_3$ and $L_1\simeq L_2$,
        \item $s_0$ and $s_3$ do not vanish simultaneously and
        \item $s_1$ and $s_2$ do not vanish simultaneously.
    \end{enumerate}
    Such a map has degree $\beta$ if $\beta\cdot D_i = \deg(L_i)$. For example, $(L\cdot D_i) = (1,1,1,0)$, $(S\cdot D_i) = (1,0,0,1)$ and $(E\cdot D_i) = (0,1,1,-1)$.

    Let $C_1 = \bbP^1$ with homogeneous coordinates $[x_0\colon x_1]$ and let $C_2 = \bbP^1$ with homogeneous coordinates $[y_0\colon y_1]$. We follow the convention that $0=[1\colon 0]$. Let $C$ be the rational nodal curve obtained by gluing of $C_1$ and $C_2$ along the origins. By abuse of notation, we denote the node of $C$ by 0.
    
    Consider the affine line $\bbA^1$ with coordinate $t$ and let $W=C\times \bbA^1$, which we view as a trivial family of nodal curves over $\bbA^1$. We denote by $W_t$ and $0_t$ the fibre and the node over $t\in\bbA^1$, respectively.  Furthermore, we choose two disjoint sections of the projection $W\to \bbA^1$ with the condition that they factor through $(C_1\setminus 0)\times \bbA^1$. These will be our marked points.

    Consider the following morphism $f\colon W \to \Bl_0 \bbP^2$, whose components we denote by $s_0,\ldots, s_3$;
    \begin{enumerate}
        \item $f$ restricted to $C_1\times  \bbA^1$ is given by
        \[
            f([x_0\colon x_1],t) = [x_0^2\colon 0\colon 2\colon x_1^2 - tx_1x_0]
        \]
        \item and $f$ restricted to $C_2\times \bbA^1$ is given by
        \[
            f([y_0\colon y_1],t) = [y_0^2\colon y_1y_0\colon y_1^2-3y_1y_0+2y_0^2\colon 0]
        \]
    \end{enumerate}
    One can check that the two descriptions glue along $0\times \bbA^1$ and satisfy non-degeneracy, so $f$ is a family in $\overline{\ccM}_{0,2}(\Bl_0 \bbP^2,2L)$ which restricts to a family of degree $2S$ on $C_1\times \bbA^1$ and of degree $2E$ on $C_2\times \bbA^1$.

    For each $t$, \Cref{prop: description css and c_X} says that the morphism $f_t\colon W_t \to \Bl_0 \bbP^2$ lies in the closed substack $\cssgn$ of $\overline{\ccM}_{0,2}(\Bl_0 \bbP^2,2L)$ where $c_{\Bl_0\bbP^2}$ is defined if and only if
    \[
        \ord_{0_t}(s_3\mid_{C_1\times \{t\}}) \geq -\beta_{C_1\times \{t\}}\cdot D_3 = -2\ E\cdot E = 2.
    \]
    Since $s_3 = x_1^2-tx_1x_0$, the condition is satisfied if and only if $t=0$. Therefore, $f$ is indeed a family of stable maps whose general fibre does not lie in $\cssgn$ but whose special fibre lies in $\cssgn$.

    As a sanity check for \Cref{prop: description css and c_X}, we show that $f_0$ is indeed the only fibre that lies in $\cssgn$ using \Cref{eq: diagram that defines c_X} directly. For that, we choose the closed embedding $\iota\colon \Bl_0 \bbP^2\hookrightarrow\bbP^2\times \bbP^1$ described in \Cref{ex: embeddings} and show that $c_{\bbP} \circ \imap (f_t)$ lies in the image of $\ccQ_{0,2}(\Bl_0 \bbP^2,2L)$ if and only if $t=0$.

    Following the definition of $\iota$ in \Cref{eq: embedding_blowup_P2xP1}, we have the following description of $\imap(f)$:
    \begin{enumerate}
        \item on $C_1\times  \bbA^1$
        \[
            \imap(f)([x_0\colon x_1],t) = [x_0^2\colon 0\colon 2 (x_1^2 - tx_1x_0)],[0\colon 2]
        \]
        \item and, on $C_2\times \bbA^1$, 
        \[
            \imap(f)([y_0\colon y_1],t) = [y_0^2\colon 0\colon 0],[y_1y_0\colon y_1^2-3y_1y_0+2y_0^2]
        \]
    \end{enumerate}

    Next, we apply $c_\bbP$, where $\bbP = \bbP^2\times\bbP^1$. This means we contract $C_2\times \bbA^1$ and twist by $\iota_\ast(2E) = (0,2)$. The result is the morphism 
    \[
        g = c_{\bbP} \circ \imap (f) \colon C_1\times \bbA^1\to \Bl_0 \bbP^2\
    \]
    given by
    \[
        g([x_0\colon x_1],t) = [x_0^2\colon 0\colon 2 (x_1^2 - tx_1x_0)],[0\colon 2x_1^2].
    \]

    Observe that if we choose coordinates $[z_0\colon z_1\colon z_2]$ on $\bbP^2$ and $[w_0\colon w_1]$ on $\bbP^1$, then the image of $\iota$ (see \Cref{eq: embedding_blowup_P2xP1}) is the closed with equation $z_1w_1-z_2w_0$. The same equations determine if a stable map to $\bbP^2\times \bbP^1$ factors through $\Bl_0 \bbP^2$, but the situation is different for quasimaps. In fact, $g$ satisfies this equation, but we argue next that $g_t$ lies in the image of $\ccQ_{0,2}(\Bl_0 \bbP^2,2L)$ if and only if $t=0$.

    The image of a family of quasimaps $(C,L_i,s_i,c_m)$ along 
    \[
        \ibar\colon \ccQ_{0,2}(\Bl_0 \bbP^2,2L)\to \ccQ_{0,2}(\bbP^2\times \bbP^1,(2,2))
    \]
    can be expressed in homogeneous coordinates as
    \[
        [s_0\colon s_1s_3\colon s_2s_3],[s_1\colon s_2].
    \]
    In particular, for $g_t$ to lie in the image of $\ibar$ we must have that the rational function
    \[
        \frac{2 (x_1^2 - tx_1x_0)}{2x_1^2}
    \]
    is regular (and, in fact, constant) on $C_1\times \{t\}$. This happens if and only if $t=0$.
\end{example}

The degree of a basepoint has the following nice geometric interpretation for maps of the form $c_X(f)$ with $f\in\css$.

\begin{remark}\label{rmk: geometric_interpretation_degree}
    The following observation can be deduced from \Cref{prop: description css and c_X}, or directly from \Cref{constr: c_X}. For every stable map $f\colon C\to X$ and every rational tail $T\subset C$ of $f$, the point $\p = f(T)$ is a basepoint of the quasimap $c_X(f)$ of degree $\beta_{\p,c_X(f)} = \deg(f\mid_T)$. Since the notion of degree of a basepoint is instrinsic, the same must hold for any other stable map $f'$ with $c_X(f') = c_X(f)$. In other words, we can reinterpret the degree $\betap$ of a basepoint $\p$ of a quasimap $\q$ as the degree of any rational tail $T$ of a stable map $f\colon C \to X$ such that $c_X(f) = \q$ and $c_X$ contracts $T$ to $\p$, whenever one such map exists. For example, the existence of such $f$ is guaranteed if $X$ is Fano, since then $c_X$ is surjective by \Cref{thm: surjectivity Fano}.
\end{remark}

\begin{example}
    We highlight the relation, explained in \Cref{rmk: geometric_interpretation_degree}, between the degree of a basepoint and the degree of rational tails contracted to it. 
    
    Let $C=\bbP^1$ with homogeneous coordinates $[s_0\colon s_1]$, consider $\bbA^1$ with coordinate $t$ and let $W=C\times \bbA^1$. Consider the following morphism $f\colon W\to \bbP^2$
    \[
        f([s_0\colon s_1],t) = [s_0\colon ts_0\colon s_1].
    \]
    Then $f$ induces a morphism $\widetilde{f}\colon W\setminus W_0 \to \Bl_0\bbP^2$, which extends to a family of quasimaps $\q$ on $W$,
    \[
        \q([s_0\colon s_1],t) = [s_0\colon ts_0\colon s_1\colon 1]),
    \]
    with a basepoint at the point $t=s_1=0$, the origin in $W_0$. 

    On the other hand, $\widetilde{f}$ can also be extended to a morphism
    \[
        \widetilde{f} \colon \widetilde{W} = \Bl_{s=t=0} W \to \Bl_0\bbP^2,
    \]
    which we view as a family of maps over $\bbA^1$. If we view $\widetilde{W}$ as the closed subvariety cut-out by $u_1t-s_1u_0$ inside $C\times \bbA^1 \times \bbP^1_{[u_0\colon u_1]}$, then the expression of $\widetilde{f}$ is
    \[
        \widetilde{f}([s_0\colon s_1],t,[u_0\colon u_1]) = [u_1s_0\colon u_0s_0 \colon u_1\colon u_1s_1]. 
    \]

    The special fibre $\widetilde{W}_0$ has two components $T = V(s_1)$ and $\widetilde{C} = V(u_0)$. The restriction of $\widetilde{f}$ to each of them is
    \[
        \widetilde{f}\mid_{T}( [u_0\colon u_1]) = [u_1\colon u_0\colon u_1\colon 0],
    \]
    of degree $E$, and
    \[
        \widetilde{f}\mid_{\widetilde{C}}( [s_0\colon s_1]) = [s_0\colon 0\colon 1\colon s_1],
    \]
    of degree $S$.

    \Cref{rmk: geometric_interpretation_degree} tells us that we can read that $\deg(\widetilde{f}\mid_{T}) = E$ directly from $\q$, by checking that $\betap = \beta_{\p,\q_0} = E$, where $\p$ is the point $s=t=0$ in $W$. We check the claim in this particular example. 

    For that, we need to compute $\betap$ following \Cref{proposition: existence choice of sigma general case}. Firstly, 
    \[
        \q_0([s_0\colon s_1]) = [s_0\colon 0\colon s_1\colon 1])
    \]
    has the following orders of vanishing at $\p$: $(\ord_\p(s_i)) = (0, \infty, 1, 0)$, and $\vanishp = \{\rho_1\}$. Let $\sigma_{i,j}$ denote the cone in \Cref{fig: fan blowup2} spanned by $\rho_i$ and $\rho_j$. The only cones $\sigma\in\Sigma(2)$ such that $\vanishp\subseteq \sigma$ are $\sigma_{0,1}$ and $\sigma_{1,3}$. The former cone gives
    \begin{align*}
        \beta(\p,\sigma_{0,1})\cdot D_2 &= 1,\\
        \beta(\p,\sigma_{0,1})\cdot D_3 &= 0.
    \end{align*}
    Therefore $\beta(\p,\sigma_{0,1}) = L$, but it does not satisfy the condition in \Cref{proposition: existence choice of sigma general case} because
    \[
        0 = \ord_\p(s_0) \not \geq \beta(\p,\sigma_{0,1})\cdot D_0 = L\cdot L = 1.
    \]
    
    Instead, we must take $\sigma_{1,3}$, which gives
    \begin{align*}
        \beta(\p,\sigma_{1,3})\cdot D_0 &= 0,\\
        \beta(\p,\sigma_{1,3})\cdot D_2 &= 1.
    \end{align*}
    Therefore $\beta(\p,\sigma_{1,3}) = E$, and one easily checks that
    \begin{align*}
        \infty = \ord_\p(s_1) &\geq \beta(\p,\sigma_{1,3})\cdot D_1 = E\cdot S = 1,\\
        0 = \ord_\p(s_3) &\geq \beta(\p,\sigma_{1,3})\cdot D_3 = E\cdot E = -1.
    \end{align*}
    So $\betap = \beta(\p,\sigma_{0,1}) = E$, as claimed.
\end{example}

\section{Surjectivity for Fano targets}\label{sec: surjectivity Fano}

We have introduced the contraction morphism $c_X$ between stable maps and stable quasimaps for smooth projective toric varieties in \Cref{constr: c_X}. In this section, we prove that $c_X$ is surjective if the target $X$ is Fano.

\begin{theorem}\label{thm: surjectivity Fano}
    Let $X$ be a smooth Fano toric variety. The contraction morphism
    \[
        c_X\colon \cssgn\to \ccQ_{g,n}(X,\beta)
    \]
    defined in \Cref{constr: c_X} is surjective. 
\end{theorem}

The proof of \Cref{thm: surjectivity Fano} is delayed until \Cref{subsec: proof surjectivity}. Before that, we introduce some terminology about the monoid of effective curve classes in \Cref{subsec: factorizations} and an important construction in \Cref{subsec: grafting}.

We fix a smooth toric Fano variety $X$ for the rest of \Cref{sec: surjectivity Fano}. \\

\subsection{Factorizations of curve classes}\label{subsec: factorizations}

\begin{definition}
    Let $\beta$ be a non-zero effective curve class in $X$. A \textit{factorization} of $\beta$ is an expression $\beta = \beta_1 + \ldots + \beta_k$ such that $k\geq 2$ and $\beta_i$ is non-zero and effective for every $i$. We say that $\beta$ is \textit{irreducible} if it does not admit any factorization.
\end{definition}

\begin{definition}\label{def:length}
    The \textit{length} of a curve class $\beta\in A_1(X)$ is 
    \[
        \lambda(\beta) = \deg(\iota_\ast \beta) = \beta \cdot (-K_X) =  \sum_{\rho \in \Sigma(1)} \beta \cdot D_\rho,
    \]
    with $\iota\colon X\hookrightarrow \bbP^N$ the anticanonical embedding. 
\end{definition}

\begin{remark}\label{rmk: length_properties}
The length $\lambda$ defines a morphism of semigroups with unit $\lambda\colon A_1(X) \to \bbZ$. Moreover, 
    \begin{enumerate}
        \item if $\beta$ is effective then $\lambda(\beta) \geq 0$ with equality if and only if $\beta = 0$ and
        \item if $\beta$ is non-zero and effective and if $\beta = \beta_1 + \ldots + \beta_k$ is a factorization, then $\lambda(\beta) > \lambda(\beta_i)$ for all $i$.
    \end{enumerate}
\end{remark}

\subsection{Grafting trees on to quasimaps}\label{subsec: grafting}

We explain how to replace a basepoint of a quasimap with a rational curve and how to extend the quasimap there given certain extra data. This construction is fundamental for the proof of \Cref{thm: surjectivity Fano}.

\begin{construction}\label{constr: main construction for surjectivity}
    Let $\q = (C,L_\rho, s_\rho,c_m)$ be a quasimap to $X$ of degree $\beta$. Suppose for simplicity that $\q$ has a unique basepoint $\p$ and let $\betap$ be the degree of $\q$ at $\p$. Let $\qreg = (C,L'_\rho, s'_\rho,c'_m)$ denote the regular extension of $\q$, constructed in \Cref{def: regular extension}.

    Given a rational irreducible curve $T\simeq \bbP^1$, a point $\p_T\in T$ and sections $t_\rho\in H^0(T,\ccO_{T}(\betap\cdot [D_\rho]))$ such that $t_\rho(\p_T) = s'_\rho(\p)$ for all $\rho \in \Sigma(1)$, we can construct a new quasimap
    \[
        \qext = (\tilde{C},\tilde{L}_\rho, \tilde{s}_\rho,\tilde{c}_m)
    \]
    of class $\beta$ by \textit{grafting $T$ on to $\q$ at $\p$} as follows:
    \begin{itemize}
        \item The curve $\tilde{C}$ is obtained by gluing $C$ and $T$ along $\p\in C$ and $\p_T\in T$. By abuse of notation, we denote the resulting node in $\tilde{C}$ also by $\p$. 
        
        \item For each $\rho \in \Sigma_X(1)$, the line bundle $\tilde{L}_\rho$ on $\tilde{C}$ is obtained by gluing the line bundle $L'_\rho$ on $C$ and the line bundle $\ccO_{T} (\betap \cdot [D_\rho])$ on $T$. 

        \item The sections $\tilde{s}_\rho$ of $\tilde{L}_\rho$ are obtained by gluing the sections $t_\rho$ and $s'_\rho$. 
        
        \item Finally, for each $m\in M$, the isomorphisms $\tilde{c}_m$ on $\tilde{C}$ are obtained by gluing the isomorphisms $c'_m\colon \otimes_{\rho\in\Sigma(1)} {L'}_\rho^{\otimes \langle m,u_\rho\rangle} \simeq \ccO_C$ and the isomorphisms $\psi_{m,\betap} \colon \ccO_{T}(\sum_{\rho\in\Sigma(1)}\ \betap\cdot [D_\rho]\ \scalar{m,u_\rho}) \simeq \ccO_{T}$ analogous to those constructed in \Cref{rmk: compatible isos}. 
    \end{itemize}
    The defining data of $\qext$ automatically satisfies generically non-degeneracy. This is clear on $C$, while on $T$ it follows from the fact that $\qreg$ is non-degenerate at $\p$. 
\end{construction}

\begin{remark}
    With the notations of \Cref{constr: main construction for surjectivity}, the regular extension $\qreg$ of a quasimap $\q$ has the following property: a section  $s_\rho$ is identically zero on an irreducible component $C'$ of $C$ if and only if $s'_\rho$ is identically zero on $C'$. 
    
    It follows that, for each basepoint $\p$ of $\q$, we can find sections $t_\rho$ satisfying the assumptions of \Cref{constr: main construction for surjectivity}. Indeed, we can choose $t_\rho$ as follows
    \begin{itemize}
        \item if $\betap\cdot [D_\rho] < 0$, we must take $t_\rho=0$ and the gluing condition holds because $s'_\rho(x)=0$ by \Cref{cor: expression regular extension},
        \item if $\betap\cdot [D_\rho] = 0$ then we must take $t_\rho$ constant with value $s'_\rho(\p)$ and
        \item if $\betap\cdot [D_\rho] > 0$ then we can choose any section $t_\rho\in H^0(T,\ccO_{T}(\betap\cdot [D_\rho]))$ with $t_\rho(\p_T)=s'_\rho(\p)$.
    \end{itemize}
\end{remark}

\begin{remark}\label{rmk: grafting prunning}
    Grafting can be undone. With the notations of \Cref{constr: main construction for surjectivity}, we can recover $\q$ from $\qext$ by contracting $T$, restricting the rest of the data to $C$ and twisting with the degree $\beta_{T}$ of $\qext$ on $T$ analogously to \Cref{prop: description css and c_X}. In that case, we say that $\q$ is obtained by \textit{pruning $\qext$ along $T$}. By \Cref{prop: description css and c_X}, pruning all the rational tails of a stable map $f\colon C \to X$ we recover the quasimap $c_X(f)$.
\end{remark}

\begin{remark}\label{rmk: extend_grafting}
    \Cref{constr: main construction for surjectivity} can be extended to quasimaps $\q$ with more than one basepoints. The grafting of $T$ on to $\q$ at a basepoint $\p$, is obtained by replacing in \Cref{constr: main construction for surjectivity} the map $\qreg$ by the quasimap obtained by extending $\q$ only at $\p$. The description of such quasimap can be obtained using \Cref{cor: expression regular extension}, replacing $B$ in the formulas by the singleton $\{\p\}$. Note that if we graft a quasimap $\q$ along all its basepoints $\p_1,\ldots, \p_\ell$, the basepoints of the resulting quasimap $\qext$ must lie on the grafted rational curves $T_1,\ldots, T_\ell$.
\end{remark}

\subsection{Proof of surjectivity}\label{subsec: proof surjectivity}

\begin{proof}[Proof of theorem \ref{thm: surjectivity Fano}]
    The morphism $c_X\colon \cssgn\to \ccQ_{g,n}(X,\beta)$ is locally of finite type, therefore by \cite[Lemma 0487]{stacks-project} it is enough to show that $c_X$ is surjective on closed points. 
    
    Let $\q = (C,L_\rho, s_\rho,c_m)$ be a 
    stable quasimap to $X$ in $\Qgn(X,\beta)$. We show by induction on $\lambda(\beta)$ (see \Cref{def:length}) that there exists a stable map $f$ to $X$ in $\Mgn(X,\beta)$ with $c_X(f)=\q$.

    If $\lambda(\beta) = 0$, then there are no basepoints by \Cref{prop: properties degree of a basepoint}. This means that $\q$ is itself a stable map and there is nothing to show.

    If $C$ has basepoints $\p_1,\ldots, \p_\ell$, grafting an irreducible rational curve $T_i$ at each basepoint $\p_i$ induces an equality of curve classes
    \begin{equation}\label{eq: factorization ell tails}
        \beta = \sum_{i=1}^\ell \beta_{\p_i} + (\beta - \sum_{i=1}^\ell \beta_{\p_i}),
    \end{equation}
    where $\beta_{\p_i}\neq 0$ for all $i$ by \Cref{prop: properties degree of a basepoint}. If $\ell \geq 2$, \Cref{eq: factorization ell tails} is a factorization, and we can use \Cref{rmk: length_properties} to conclude by induction. 
    
    Therefore, we can restrict to the case that $\lambda(\beta) > 0$ and $C$ has only one basepoint $\p$. Furthermore, by restricting ourselves to the irreducible component containing the basepoint $\p$, we can assume that $C$ is irreducible.  Let $\qext$ be obtained by grafting an irreducible rational curve $T_1$ to $\q$ at $\p$. Then we get an expression
    \[
        \beta = \betap + (\beta - \betap)
    \]
    with $\betap\neq 0$. If $\beta - \betap$ is also non-zero, we conclude by induction on $\lambda$. Thus, we can assume that $\beta = \betap$. 
    
    In other words, we only need to deal with the particular case that on the new curve $C\cup T_1$ the degree is 0 on $C$ and $\beta$ on $T_1$. The strategy is to continue grafting irreducible rational curves to $\qext$ at the new basepoints, which necessarily lie on $T_1$ (see \Cref{rmk: extend_grafting}). Again, the only problematic case is if there is a unique basepoint $x_1$ of $\qext$ on $T_1$ such that $\beta_{x_1} = \beta$, because we could enter a loop and end up with an infinite chain of $\bbP^1$'s. To rule this out, we use \Cref{lemma: key lemma} as follows: when grafting $T_1$ to $\q$, we are allowed to choose the sections $t_\rho$ on $T_1$ (satisfying the conditions in \Cref{constr: main construction for surjectivity}). If we show that we can choose the sections $t_\rho$ with the property that
    \begin{equation}\label{eq: surjectivity win condition}
        \text{there is } t_\rho \neq 0 \text{ such that } \ord_{x_1} (t_\rho) < \beta \cdot D_\rho,
    \end{equation}
    then \Cref{lemma: key lemma} (see also \Cref{rmk: grafting prunning}), applied to $C_1 = C\cup T_1$ and $C_2 = T_2$, ensures that when we graft $T_2$ we cannot have $\beta_{T_1} = 0$ and $\beta_{T_2} = \beta$, and so we conclude by induction.

    To get \eqref{eq: surjectivity win condition} for a specific class $\beta$, it suffices to show one of the following:
    \begin{equation}\label{item: win condition 2}
        \text{there exists } \rho \in \Sigma(1) \text{ such that } s_\rho(\p)=0, s_\rho\neq 0 \text{ and } \beta\cdot D_\rho \geq 2,
    \end{equation}
    or
    \begin{equation}\label{item: win condition 11}
         \text{ there exist } \rho \neq \rho' \in \Sigma(1) \text{ such that } \beta\cdot D_\rho = \beta\cdot D_{\rho'} =1 \text{ and at most one is } 0 \text{ on } \compp.
    \end{equation}
    In both cases, it is clear we can choose $t_\rho$ on $T_1$ to have simple disjoint zeroes, which ensures \eqref{eq: surjectivity win condition}.

    To apply this strategy, we distinguish two cases, depending on whether $\beta$ is irreducible or not. In the latter case we shall use that $X$ is Fano, which is not needed in the former case.

    Firstly, we assume that $C$ is irreducible and has a unique basepoint $\p$, that $\beta$ is irreducible with $\lambda(\beta) > 0$ and that $\betap = \beta$. Since $\beta$ is irreducible, it can be represented by the closure of the toric orbit associated to a wall $\tau \in \Sigma$; that is, $\beta = [V(\tau)]$. 
    By \cite[Proposition 6.4.4]{CLS}, there exist two distinct rays $\rho_1,\rho_2\in \Sigma(1)$ such that $\beta \cdot D_{\rho_i} = 1$ for $i=1,2$.
    By \Cref{lem: exists nonvanishing section}, there is $\rho \in \Sigma(1)$ such that $s_\rho(\p) = 0$, $s_\rho\neq 0$ and $\beta \cdot D_\rho > 0$.
    If $\beta \cdot D_\rho \geq 2$, we are done by \eqref{item: win condition 2}.
    Otherwise,
    either $\{\rho,\rho_{1}\}$ or $\{\rho,\rho_{2}\}$
    satisfy \eqref{item: win condition 11}.
        
    Finally, we deal with the case where $\beta$ is not irreducible. We also assume, as before, that $\lambda(\beta)>0$,  
    that $C$ is irreducible, that there is a unique basepoint $\p\in C$ and that $\beta = \betap$. 

    By \Cref{lem: exists nonvanishing section}, there is $\rho \in \Sigma(1)$ such that $s_\rho(\p) = 0$, $s_\rho\neq 0$ and $\beta \cdot D_\rho > 0$.

    We number the rays in $\Sigma(1)$ as $\rho= \rho_1, \rho_2,\ldots, \rho_r$. Since $X$ is Fano, the anticanonical divisor $-K_X = \sum_{i=1}^r D_{\rho_i}$ is very ample (by \cite[Theorem 6.1.15]{CLS}).  Let $\iota\colon X \hookrightarrow \bbP^N$ be the corresponding closed embedding. Since $\beta$ is not irreducible, let $\beta = \beta_1+\beta_2$ be a factorization. Then
    \begin{equation}\label{eq: use Fano in surjectivity}
        \sum_{i=1}^r \beta \cdot D_{\rho_i} = \deg(\iota_\ast \beta) = \deg(\iota_\ast \beta_1) + \deg(\iota_\ast \beta_2) \geq 2.
    \end{equation}
    It follows that either $\beta \cdot D_\rho \geq 2$, and then we conclude by \eqref{item: win condition 2}, or there is $i\neq 1$ such that $\beta \cdot D_{\rho_1} \geq 1$, and then $\{\rho, \rho_i\}$ satisfy \eqref{item: win condition 11}.

    As a result of the induction argument, we end up with a prestable map $f$ to $X$ whose contraction is the original quasimap $\q$. If  $f$ is not stable, we can stabilize it (by contracting its irreducible components of degree 0 with no marks). The result is a stable map $f$ which lies in $\cssgn$ by construction. More precisely, one only needs to check the condition in \Cref{prop: description css and c_X}, and it follows from the following observation: in \Cref{constr: main construction for surjectivity}, we have that
    \[
        \ord_\p(\tilde{s}_\rho\mid_C) + \beta_{T_i}\cdot [D_\rho] = \ord_\p(s'_\rho\mid_C) + \betap\cdot [D_\rho] = \ord_\p(s_\rho) \geq 0.
    \]
    It is also clear that $f$ has degree $\beta$ since the total curve class is preserved by grafting.
\end{proof}

We conclude with two lemmas used in the proof of \Cref{thm: surjectivity Fano}.

\begin{lemma}\label{lemma: key lemma}
    Let $\q$ be a quasimap with underlying curve $C$ consisting of two irreducible components $C_1$ and $C_2$ glued along a point $\p$. Suppose $C_2$ has genus 0.
    Let $\qone$ and $\qtwo$ be the restrictions of $\q$ to $C_1$ and $C_2$ respectively. Assume $\beta_{C_1}=0$ and $\beta_{C_2}\neq 0$. Let $\qpr$ be the quasimap obtained
    by pruning $\q$ along $C_2$.
    Then for every $\rho\in\Sigma(1)$ such that $s'_\rho\neq 0$, or equivalently such that $s_\rho\mid_{C_1}\neq 0$, the following holds
	\[
		\ord_\p(s'_\rho) = \beta_{C_2}\cdot D_\rho = \beta'\cdot D_\rho.
	\]
\end{lemma}

\begin{proof}
    By definition, 
    \[
        s'_{\rho} =s_\rho\mid_{C_1} \cdot \coord^{\beta_{C_2}\cdot D_\rho},
    \]
    with $\coord$ a local coordinate at $\p$ inside $C_1$. Therefore, $s'_\rho= 0$ if and only if $s_\rho\mid_{C_1}= 0$. Furthermore, if $s_\rho\mid_{C_1}\neq 0$ then $\ord_\p(s_\rho\mid_{C_1}) = 0$ because $\beta_{C_1}=0$, and $\ord_\p(s'_\rho) = \beta_{C_2}\cdot D_\rho= \beta'\cdot D_\rho$. 
\end{proof}

\begin{lemma}\label{lem: exists nonvanishing section}
    Let $\q = (C,L_\rho, s_\rho,c_m)$ be a 
    stable quasimap to $X$ of degree $\beta$ with $C$ irreducible and let $\p$ be a basepoint of $\q$. Then there exists $\rho \in \Sigma(1)$ such that $s_\rho(\p) = 0$ but $s_\rho \neq 0$. In particular, $\beta \cdot D_\rho > 0$. 
\end{lemma}

\begin{proof}
    Since $\p$ is a basepoint, by \Cref{eq: Z Sigma primitive collections} there is a primitive collection $\ccP\subseteq \Sigma(1)$ such that $s_{\rho}(x) = 0$ for all $\rho\in \ccP$. By generic non-degeneracy, there exists $\rho \in \ccP$ such that $s_{\rho}$ is not identically 0. The existence of $s_\rho$ ensures that $\beta \cdot D_\rho = \deg(L_{\rho})> 0$. 
\end{proof}

\begin{remark}\label{rmk: relax Fano in surjectivity}
    \Cref{thm: surjectivity Fano} also holds for $X$ smooth projective toric with the following property: for each effective non-irreducible non-zero curve class $\beta \in A_1(X)$ such that there exists $\rho \in \Sigma(1)$ with $\beta \cdot D_\rho = 1$, there exists $\rho'\neq \rho$ such that $\beta \cdot D_{\rho'} > 0$. 

    This condition is true if $X$ is Fano by the argument involving \Cref{eq: use Fano in surjectivity}.

    A way to control this condition is to look at 
    \[
        \sum_{\rho \in \Sigma(1)} \beta \cdot D_\rho = \beta \cdot (-K_X).
    \]
    For example, the above condition holds if $\-K_X\cdot \beta \geq 2$ for every effective non-zero curve class $\beta \in A_1(X)$. But, in that case, $X$ is Fano as $-K_X$ is very ample by \cite[Theorems 6.1.15 and 6.3.22]{CLS}.
\end{remark}

\bibliographystyle{alpha}
\bibliography{bibliography}

\end{document}